\documentclass[12pt, a4paper]{amsart}

\usepackage{amssymb}
\usepackage{textcomp}
\usepackage{mathrsfs}
\usepackage{verbatim}
\usepackage[utf8]{inputenc}
\usepackage{hyperref}

\theoremstyle{plain}
\newtheorem{theorem}{Theorem}[subsection]
\newtheorem{lemma}[theorem]{Lemma}

\newtheorem{corollary}[theorem]{Corollary}
\newtheorem{proposition}[theorem]{Proposition}

\theoremstyle{definition}
\newtheorem{definition}[theorem]{Definition}
\newtheorem{example}[theorem]{Example}

\newtheorem{problem}[theorem]{Problem}
\newtheorem{remark}[theorem]{Remark}
\newtheorem{notation}[theorem]{Notation}

\theoremstyle{remark}

\newtheorem*{claim}{Claim}

\numberwithin{figure}{section}
\numberwithin{equation}{section}

\DeclareMathOperator{\id}{id}
\DeclareMathOperator{\res}{res}

\DeclareMathOperator{\tr}{tr}

\DeclareMathOperator{\GL}{\mathbf{GL}}

\DeclareMathOperator{\SO}{\mathbf{SO}}

\DeclareMathOperator{\SL}{\mathbf{SL}}
\DeclareMathOperator{\nSL}{SL} 

\DeclareMathOperator{\End}{End}
\DeclareMathOperator{\Sp}{\mathbf{Sp}}
\DeclareMathOperator{\cl}{cl}
\DeclareMathOperator{\chr}{char}

\DeclareMathOperator{\zardim}{zardim}

\DeclareMathOperator{\Ker}{Ker}

\DeclareMathOperator{\alg}{alg}
\DeclareMathOperator{\divv}{div}

\DeclareMathOperator{\dd}{d}
\DeclareMathOperator{\Ad}{Ad}
\DeclareMathOperator{\adj}{ad}
\DeclareMathOperator{\Id}{Id}
\DeclareMathOperator{\Int}{Int}
\DeclareMathOperator{\RV}{RV}
\DeclareMathOperator{\rv}{rv}
\DeclareMathOperator{\tp}{tp}

\newcommand{\ACVF}{\mathrm{ACVF}}

\newcommand{\Rr}{{\mathcal R}}
\newcommand{\M}{{\mathcal M}}

\renewcommand{\O}{{\mathcal O}}

\newcommand{\N}{{\mathbb N}}

\newcommand{\R}{{\mathbb R}}
\newcommand{\Q}{{\mathbb Q}}

\def\ldim{\mathop{\rm ldim}\nolimits}

\def\acl{\mathop{\rm acl}\nolimits}
\def\Aut{\mathop{\rm Aut}\nolimits}
\def\RV{\mathop{\rm RV}\nolimits}
\def\rv{\mathop{\rm rv}\nolimits}

\newcommand{\Klamalg}{K_\lambda^{\alg}}
\newcommand{\Klampalg}{K_{\lambda'}^{\alg}}

\begin{document}

\title{On simple groups definable in some valued fields}

\author{Jakub Gismatullin}

\address{J.~Gismatullin, Instytut Matematyczny Uniwersytetu Wroc{\l}awskiego, pl. Grunwaldzki 2, 50-384 Wroc{\l}aw, Poland}

\email{jakub.gismatullin@math.uni.wroc.pl, https://gismat.math.uni.wroc.pl/}

\thanks{\noindent The first author is supported by the National Science Centre, Poland NCN grants no.  2014/13/D/ST1/03491 and 2017/27/B/ST1/01467.}

\author{Immanuel Halupczok}

\address{I.~Halupczok, Mathematisches Institut, Heinrich-Heine-Universit\"at D\"usseldorf, Mathematisch-Naturwissenschaftliche Fakult\"at, Universit\"atsstr.\ 1, 40225 D\"usseldorf, Germany}
\email{math@karimmi.de}

\thanks{The second author is supported by the research training group
\emph{GRK 2240: Algebro-Geometric Methods in Algebra, Arithmetic and Topology},
which is funded by the DFG}

\author{Dugald Macpherson}

\address{D.~Macpherson, School of Mathematics, University of Leeds, Woodhouse Lane, Leeds, LS2 9JT, UK}

\email{H.D.MacPherson@leeds.ac.uk}

\thanks{The third author is supported by the Engineering and Physical Sciences Research Council grant UKRI2391.}

\date{\today}



\begin{abstract}
We prove that non-abelian definable, definably simple groups in 1-h-minimal henselian valued fields are essentially already linear algebraic groups. Here, the group is assumed to live in the home sort. We have a similar result in pure algebraically closed valued fields of positive characteristic, under the additional assumption that the definable group is a subgroup of a linear algebraic group.
\end{abstract}

\maketitle

\tableofcontents

\section{Introduction}

In the model theory of fields, the notion of {\em definable set}
(solution set of a first order formula, possibly with parameters) provides an abstraction of that of {\em constructible set}. In the context of algebraically closed fields in the language of rings the notions coincide, whilst in real closed fields in the language of ordered rings, the definable sets are exactly the semi-algebraic sets. Definability has corresponding geometric content in other model-theoretically well-understood structures such as differentially closed fields, algebraically closed fields with an automorphism, separably closed fields, ${\mathbb Q}_p$, and algebraically closed valued fields. In all such structures, definable {\em groups} play a fundamental role, as measures of model-theoretic complexity, or of the interaction between different definable sets, or through definable group actions. Understanding the definable simple groups (or {\em definably simple} groups, that is, definable groups with no proper non-trivial definable normal subgroups) is often critical.

In this paper, we classify definably simple groups in the context of Henselian valued fields of characteristic 0, and, under a linearity assumption,  algebraically closed valued fields of characteristic $p$. In characteristic 0, our methods apply also to {\em 1-h-minimal} expansions of valued fields. This broad framework includes several contexts which have had extensive recent study: the `valued fields with analytic structure' of Cluckers-Lipshitz \cite{CL.analyt} stemming ultimately from
Denef and van den Dries \cite{denef}; and $T$-convex expansions of power-bounded o-minimal expansions of fields \cite{vdDries-Lewenberg}.  See Example~\ref{ex1h} for more detail on these settings. Our main result is the following.

\begin{theorem} \label{main} 
\begin{enumerate}
\item[(i)] Let $(K,v,\ldots)$ be a 1-$h$-minimal expansion of a henselian valued field of characteristic 0. Let $G$ be an infinite non-abelian group definable  in the structure $(K,v,\ldots)$ with universe a definable subset of $K^n$ for some $n$, and suppose that $G$ has no proper infinite definable normal subgroup.
Then there is a semisimple, almost $K$-simple
and $K$-isotropic linear algebraic $K$-group $\mathbf{H}$ and a group $G^*$ definably isomorphic to $G/Z(G)$ 
 such that 
$\mathbf{H}(K)^+\leq G^*\leq \mathbf{H}(K)$. 
\item[(ii)] Let $(K,v)$ be an algebraically closed valued field of characteristic $p>0$, and let $G\leq \SL_n(K)$
be an infinite non-abelian definable group which has no proper infinite definable normal subgroup. Then there is a semisimple, almost $K$-simple and $K$-isotropic linear algebraic $K$-group $\mathbf{H}$ such that
$G/Z(G)$ is definably isomorphic to $\mathbf{H}(K)$. 
\end{enumerate}
\end{theorem}

For the relevant definitions concerning algebraic groups, see 
Section~\ref{sec:lin}, and for the relevant background on the model theory of valued fields and 1-$h$-minimality see Sections ~\ref{sec:val} and \ref{sec:hen-min}.

The conclusion in (i) says that $G^*$ is ``almost equal'' to $\mathbf{H}(K)$. Indeed, the group $\mathbf{H}(K)^+$ (Definition~\ref{def:defg+}) is Zariski-dense in $\mathbf{H}(K)$ (see Remark~\ref{rem:semdense}); and if $K$ is algebraically closed, those two groups are even equal (see Remark~\ref{rem:g+acf}).
We are not certain whether under the given hypotheses $\mathbf{H}(K)^+$ is automatically definable. If it is definable, then, since
$\mathbf{H}(K)^+$ is a normal subgroup of $\mathbf{H}(K)$,
the conclusion will be that $G^*=\mathbf{H}(K)^+$. This holds, for example, if $K$ is real closed, or is a local field (see Lemma~\ref{lem:kneser}).

\begin{remark}
 We are not certain if the theorem would be correct without first replacing $G$ by $G/Z(G)$. In particular, we factor out the centre when defining a homomorphism to $\SL_n$ in the proof of Proposition~\ref{linearityproof}.
\end{remark}

We stress two key limitations in this result.

First, in Theorem~\ref{main} (ii), which is in characteristic $p$, we have assumed that $G\leq \SL_n(K)$. The reason for this is that we use  the adjoint representation defined via Jacobian matrices in the proof of (i), to linearise the group (see Proposition~\ref{linearityproof}). As presented this method only works in characteristic 0.

Second, Theorem~\ref{main} is a statement only about definable groups in the `home' sort $K$, rather than about arbitrary interpretable groups. The sorts required for elimination of imaginaries in henselian fields are rather complicated (see \cite{hhm,hmrc,mellor,vicaria}) and our methods do not seem strong enough to handle interpretable groups. Some progress on the latter lies in the paper \cite{hrush-rid} which develops a notion of \emph{generically metastable group} interpretable in the theory $\ACVF$ of algebraically closed non-trivially valued fields. The results there  seem to be disjoint from those in this paper. There have also been several recent papers (some in the 1-$h$-minimal context) on {\em interpretable} groups and fields in valued fields, avoiding use of elimination of imaginaries, by Halevi, Hasson and Peterzil -- see e.g. \cite{halevi-fields}, \cite{halevi-groups1},
\cite{halevi-groups3}.
The results of \cite{halevi-groups3} are in more restricted classes of valued fields but complement our theorem well. The main theorem there reduces questions about an {\em interpretable} definably simple group $G$ to the case when $G$ definably embeds in $\GL_n(K)$, and our theorem then applies. In particular, Theorem~\ref{main} above, combined with \cite[Corollary 10.5]{halevi-groups3}, has the following consequence. See \cite[Section 3.4]{hk-motivic} for the notion of $V$-minimal expansion (of an algebraically closed valued field of characteristic 0).

\begin{corollary}
Let $(K,v,\ldots)$ be a power-bounded   $T$-convex field, a $V$-minimal field, or a $p$-adically closed field, and let   $G$ be an infinite non-abelian group which is interpretable in $(K,v,\ldots)$ and is definably simple. Then either $G$ is definably isomorphic to a subgroup of $\GL_n(k)$ where $k$ is the residue field, or $G$ is definably isomorphic to $\mathbf{H}(K)^+$ where $\mathbf{H}$ is a semisimple almost $K$-simple and $K$-isotropic linear algebraic $K$-group. 
\end{corollary}
\begin{proof} The only point that needs noting is that in these cases, the underlying field is algebraically closed, real closed, or $p$-adically closed, and so $\mathbf{H}(K)^+$ is definable, by 
Lemma~\ref{lem:kneser}(ii) below. Note also that all of the structures considered here are $1$-h-minimal, by \cite[Corollary~6.2.7, Theorem~6.3.4, Proposition~6.4.2]{iCR.hmin}.
\end{proof}

We also note a result in the $p$-adic case related to Theorem~\ref{main}, pointed out to us by Caprace: In the case $K = \mathbb{Q}_p$, by \cite[Lemma 3.8]{pill-fields} (see also \cite{AH}), our definable group $G$ naturally carries the structure of a $p$-adic Lie group, so our result can be considered as a statement about definable, definably simple $p$-adic Lie groups.
On pp. 77-78 of \cite{willis}, a similar result is given in a non-definable setting: if $G$ is a topologically simple $p$-adic Lie group whose adjoint representation is non-trivial, then $G$ is isomorphic to a group of the form $\mathbf{H}({\mathbb{Q}}_p)^+/Z$, where $\mathbf{H}$ is a non-abelian almost $K$-simple $K$-isotropic linear algebraic $K$-group with $K=\mathbb{Q}_p$, and $Z$ is the centre of $\mathbf{H}({\mathbb{Q}}_p)^+$. Note that in the definable setting, the adjoint representation is automatically non-trivial as soon as $G$ is non-abelian; see Proposition~\ref{linearityproof}.

We sketch the structure of the proof of Theorem~\ref{main}. There is a well-behaved notion of (topological) dimension for definable sets. A key first step, in the characteristic 0 1-$h$-minimal setting, is to endow the definable group $G$ with a definable $C^1$-manifold structure. For this, we use results from  \cite{AH}, who develop the notion of a {\em definable weak strictly differentiable} (dwsd) manifold, and show that any definable group can be definably equipped with  structure so that it becomes a dwsd Lie group.  The method has precedents in the work of Wencel 
\cite{Wen.topGrp}, itself based on a proof of Pillay \cite{pill-omin} and ultimately on  Weil's `group-chunks' argument. Given the dwsd Lie structure, the Jacobian representation then gives a homomorphism
$\varphi\colon G \to \SL_n(K)$, and a dimension argument using definability ensures that ${\rm Im}(\varphi)$ is not trivial, and hence $\Ker(\varphi)$ is finite and equal to the centre $Z(G)$. We thereby replace $G$ by the definable subgroup $G/Z(G)$ of $\SL_n(K)$. This step is not needed in the characteristic $p$ case
(Theorem~\ref{main}(ii)), where linearity is assumed.

The second stage works more directly when $(K,v)$ is a \emph{pure} valued field, \emph{i.e.} does not carry definable extra (e.g. analytic) structure. Slightly over-simplifying, we take the Zariski closure of $G$ in $\SL_n(K)$, and show that this has the form $\mathbf{H}(K)$ for some semisimple almost $K$-simple $K$-isotropic algebraic group $\mathbf{H}$ defined over $K$. 
By dimension arguments, $G$ is valuation-open in $\mathbf{H}(K)$. These dimension arguments are invalid in the more general case when $(K,v,\ldots)$ is equipped with additional 1-$h$-minimal structure, and are replaced in the latter case by a Lie algebra argument: we take the Lie algebra $L(G)$ of the definably almost simple linear group $G$, argue that this is a simple Lie algebra, and that $G$ is embedded   in the automorphism group of the Lie algebra (in place of the Zariski closure of $G$), which has the same dimension as $G$. Thus, roughly speaking, the role of $\mathbf{H}(K)$ is replaced by $\Aut(L(G))$. This last argument uses ideas from \cite{pps}.

The third stage is to show that $G$ cannot be a \emph{bounded} subgroup of $\mathbf{H}(K)$, that is, there is no $\gamma\in \Gamma$ (the value group of $K$) such that all matrix entries of all elements of $G$ have value greater than $\gamma$. The argument here uses ideas from Bruhat-Tits buildings, though these are not mentioned explicitly.  We show that if $G$ were bounded then it would embed in a variant $\nSL_n(\eta)$ of $\SL_n(\mathcal{O})$, where $\O$ is the valuation ring. The  group $\SL_n(\mathcal{O})$ has a chain of normal subgroups with trivial intersection (the \emph{congruence subgroups}) which is  incompatible with definable almost simplicity of $G$, and a similar argument works in $\nSL_n(\eta)$.

Finally, we use variants (for Krull valuations) of two results from Prasad \cite{prasad}, to which he gives earlier attributions. First, by the last paragraph the group $\mathbf{H}(K)$ cannot be bounded, so 
a result in \cite{prasad} yields that $\mathbf{H}$ is $K$-isotropic, which ensures that $\mathbf{H}(K)^+$ is almost simple and Zariski dense in $\mathbf{H}(K)$. Second, since $G$ is valuation-open in $\mathbf{H}(K)$ and not bounded, applying \cite{prasad} with some finessing we obtain $G\geq \mathbf{H}(K)^+$. If $\mathbf{H}(K)^+$ is definable then we have equality. Some additional work is involved in these applications of \cite{prasad}, largely because we are dealing with arbitrary Krull valuations, rather than just valuations with archimedean value group. The biggest difficulties arise when $\Gamma$ has no maximal proper convex subgroup.

Some intermediate results in the paper may have independent interest and further applications. See Theorems~\ref{2.34}, \ref{simpleliealg} and Proposition~\ref{linearityproof} at the end of Section 3, concerning Lie algebras of definable groups in characteristic 0 in the 1-$h$-minimal context, extending material from \cite{AH}. Also Theorem~\ref{thm:prasadii} is a variant of a result from \cite{prasad} attributed there to Tits -- in some ways it is a strengthening of the Tits result, but it has an extra definability assumption. Proposition~\ref{prop:main} is a generalisation of the fact that an infinite bounded subgroup of $\SL_2(\mathbb{Q}_p)$  cannot be simple -- the latter special case is easily seen by considering the action of $\SL_2(\mathbb{Q}_p)$ on a $(p+1)$-branching tree.

We mention briefly some further literature on definable groups in valued fields. 
First, Pillay showed in \cite{pill-fields} that any infinite \emph{field} definable in $\Q_p$ is a finite extension of $\Q_p$. He also investigated in
\cite{pill} definable subgroups $G$ of $\GL_n(\Q_p)$. Using an argument reminiscent of Zilber Indecomposability, he shows that the commutator subgroup $[G,G]$ is Zariski closed, and that $G$ is simple as an abstract group if and only if it has no proper non-trivial Zariski-closed normal subgroup. In the paper \cite{hrush-pill}, among other results it is shown that if $G$ is a definable group in the home sort in ${\mathbb Q}_p$, then an open subgroup of $G$ is definably isomorphic to an open subgroup of $\mathbf{H}({\mathbb{Q}}_p)$ for some connected algebraic ${\mathbb Q}_p$-group $\mathbf{H}$.
A 2015 PhD thesis of Druart contains extensive results on \emph{commutative} definable subgroups of $\GL_n(\Q_p)$ and (see \cite{druart}) on definable subgroups of $\SL_2(K)$ where $K$ is $p$-adically closed. We also  mention a number of recent papers of interest of Ningyuan Yao, Will Johnson,  and co-authors. For example, tackling a question from \cite{pill},  in \cite{yao} it is shown that if $G$ is an algebraic group over $\mathbb{Q}_p$ which is commutative or reductive, then any open subgroup of $G(\mathbb{Q}_p)$ is definable in $\mathbb{Q}_p$  (equivalently, semi-algebraic).

{\bf Organisation of paper.}
The paper is structured as follows. We give background in Section \ref{sec:mod}, on linear algebraic groups, valued fields and on 1-$h$-minimality. Section 3.1 develops the background on definable manifolds
(in particular {\em dwsd manifolds}) from \cite{AH}. This concludes with Proposition~\ref{p.Gman}, due to \cite{AH} but stemming from earlier results of Wencel and previously Pillay,  which says that in a 1-$h$-minimal context, any definable group can be given definable manifold structure. This is applied in Section 3.2 to show that in a 1-$h$-minimal structure, if $G$ is a  definable definably almost simple group then its tangent space $T_e(G)$ can be given definably the structure of a simple Lie algebra of the same dimension, on which $G$ acts as a group of automorphisms.  The variants on Prasad's paper \cite{prasad2}
are presented in Section 4.1, and we show in Section 4.2 that a definable definably almost simple  group cannot be bounded in the valuation sense (Proposition~\ref{prop:main}). The proof of the main theorem is completed in Section 5.

{\bf Further directions.} We mention two further questions.

\begin{problem} \label{q-semisimple} Describe definably semisimple groups definable in the (expansions of) fields which are considered in this paper. Here a group is {\em definably semisimple} if it has no infinite abelian definable normal subgroup. A typical example of such a group should be an  open subgroup of a direct product of the groups described in Theorem~\ref{main}. Note that $\SL_n(\mathcal{O})$ (where $\mathcal{O}$ is the valuation ring of $K$) is an example of a definably semisimple definable group which is far from simple.
\end{problem}

\begin{problem} Extend Theorem~\ref{main} (at least for $\ACVF$) for expansions of fields with a definable section of the residue map (so definably simple groups will include rational points of simple groups defined over the residue field). In the case of $\ACVF$, there is a quantifier-elimination for such structures in Section 6 of \cite{poisson}.
\end{problem}

{\bf Acknowledgements:} This work has benefitted from conversations with many people over a long period. We warmly thank Patrick Simonetta -- indeed, the work began, in the context of $\ACVF$,  through discussions of the third author with Simonetta during and after the latter's Marie Curie Fellowship in Leeds in 1997--98. We thank Yatir Halevi, Assaf Hasson and Kobi Peterzil for communicating results from \cite{halevi-fields, halevi-groups1,  halevi-groups3}, and Gopal Prasad for providing additional information related to \cite{prasad}.

\section{Background} \label{sec:mod}

\subsection{Linear algebraic groups} \label{sec:lin}

We give here some background, terminology, and conventions regarding linear algebraic groups, as needed in this paper. Our terminology concerning algebraic geometry closely follows \cite{Humphreys, springer}, in particular always having some fixed algebraic closed field $\mathbf K$ in the background.
To further simplify the exposition, we will consider all linear algebraic groups as being embedded into $\SL_n = \SL_n(\mathbf K)$, for some $n$, and similarly, all varieties will be embedded into some affine space $\mathbf K^n$. Such linear algebraic groups and varieties will always be denoted by boldface letters, whereas non-boldface letters might typically denote subsets of the $K$-valued points of varieties.
We recall some details for readers not so familiar with this subject.

Fix a perfect field $K$, and also fix some algebraically closed field $\mathbf K$ containing $K$. The precise choice of $\mathbf K$ does not matter, as long as it contains all fields we are interested in; most of the time, one can simply take $\mathbf K$ to be the algebraic closure $K^{\alg}$ of $K$.

A subset $\mathbf X \subset \mathbf K^n$ is called {\em $K$-closed} (or {\em defined over $K$} or a {\em variety defined over $K$}) if it is the set of common zeros of finitely many polynomials in $n$ variables with coefficients in $K$.\footnote{In general,
``defined over $K$'' is a more restrictive condition, but since we assume $K$ to be perfect, it coincides with being $K$-closed; see \cite[AG \textsection 11]{borel}, \cite[Section 1]{borell}.}
We call $\mathbf X$ {\em irreducible over $K$} if it is not the union of two proper $K$-closed subsets. A {\em $K$-regular} function on $\mathbf X$ is the restriction to $\mathbf X$ of a polynomial in $n$ variables with coefficients in $K$.
We write $\mathbf X(K) = \mathbf X \cap K^n$ for the set of $K$-rational points of $\mathbf X$. (Note that $\mathbf X$ is automatically also defined over any field extension $L \subset \mathbf K$ of $K$, so that $\mathbf X(L)$ also makes sense for such $L$.)

The {\em Zariski topology} on $\mathbf X(K)$ is the topology whose closed sets are the sets of the form $\mathbf Z(K)$, for some $K$-closed $\mathbf Z \subset \mathbf X$.
We denote the Zariski closure in $\mathbf X(K)$ of an arbitrary subset $Y \subset \mathbf X(K)$ by $\cl_K(Y)$.

\begin{remark} \label{note-on-dim} Note that the Zariski topology on $\mathbf X(K)$ is not just the restriction of the Zariski topology on $\mathbf X = \mathbf X(\mathbf K)$ (where $\mathbf K$-closed sets are considered). Nevertheless, for $Y \subset \mathbf X(K)$, one obtains that $\mathbf Z' := \cl_{\mathbf K}(Y) \subset \mathbf X$ is $K$-closed, since any field automorphism of $\mathbf K$ fixing $K$ pointwise also fixes $Y$ and hence sends $\mathbf Z'$ to itself. In particular, $\mathbf Z'(K)$ makes sense and is equal to $\cl_K(Y)$.
\end{remark}

The {\em dimension} of a variety $\mathbf X$ can be defined in various equivalent ways,
e.g., as the {\em Zariski dimension}: the length $d$ of the longest chain $\mathbf Y_0  \subsetneq \dots \subsetneq \mathbf Y_d \subseteq \mathbf X$ of irreducible (over $\mathbf K$) varieties $\mathbf Y_i$ defined over $\mathbf K$ contained in $\mathbf X$; see e.g. the Definition after Corollary 1.6 in \cite[Ch. I]{hartshorne}
for more details. We denote that dimension 
by $\zardim \mathbf X$.

We set $\SL_n := \SL_n(\mathbf K)$. 
This is a $K$-closed subset of $\mathbf K^{n^2}$, so the above notions apply to subsets of $\SL_n$.
A {\em linear algebraic} $K$-group (often below just {\em $K$-group} for short) is a subgroup $\mathbf H \le \SL_n$ (for some $n$) which is defined over $K$.
We consider $\GL_n$ as a linear algebraic $K$-group by embedding it into $\SL_{n+1}$, sending a matrix $A \in \GL_n$ to $\begin{pmatrix}A&0\\0&(\det A)^{-1}\end{pmatrix}$.

If $\mathbf H \subset \SL_n$ and $\mathbf H' \subset \SL_{n'}$ are two $K$-groups, then a {\em $K$-morphism} $f\colon \mathbf H \to \mathbf H'$ (also called a {\em morphism defined over $K$}) is a group homomorphism from $\mathbf H$ to $\mathbf H'$ given by an $n'\times n'$-tuple of $K$-regular functions on $\mathbf H$. If $f$ is bijective and its inverse is also a $K$-morphism, we call $f$ a {\em $K$-isomorphism}, and if such an $f$ exists, we say that $\mathbf H$ and $\mathbf H'$ are {\em isomorphic over $K$}.
\begin{remark} \label{rem:bijmorph}
For any $K$-morphism $f: \mathbf H \to \mathbf H'$ of $K$-groups,
the image $f(\mathbf H) \subset \mathbf H'$ is also a $K$-group. If $f$ is injective and $K$ has characteristic $0$, then $f^{-1}\colon f(\mathbf H) \to \mathbf H$ is also a $K$-morphism and so $f$ induces a $K$-isomorphism from $\mathbf H$ to $f(\mathbf H)$ (see for example \cite[Exercise 5.3.5(1)]{springer}).
\end{remark}

We write $\mathbf{G}_a=\mathbf{G}_a(\mathbf{K})$ for the additive group of $\mathbf{K}$, and
$\mathbf{G}_m=\mathbf{G}_m(\mathbf{K})$ for the multiplicative group of $\mathbf{K}$.
Note that both of them can be considered as $K$-groups by identifying them with subgroups of $\SL_2$:
$\mathbf{G}_a = \begin{pmatrix}1&a\\0&1\end{pmatrix}$
and 
$\mathbf{G}_m = \begin{pmatrix}a&0\\0&a^{-1}\end{pmatrix}$.

Usually, linear algebraic $K$-groups $\mathbf H$ are defined abstractly, \emph{i.e.}, $\mathbf H$ does not come with a fixed embedding into $\SL_n$. To see that various notions which we define naively (using a fixed embedding) also make sense abstractly, we need to verify that the choice of the embedding does not matter. More precisely, suppose that $\mathbf H \subset \SL_n$ and $\mathbf H' \subset \SL_{n'}$ are $K$-groups and that 
$f\colon \mathbf H \to \mathbf H'$ is a $K$-isomorphism. Then:
\begin{itemize}
  \item The map $f$ sends $K$-closed subsets of $\mathbf H$ to $K$-closed subsets of $\mathbf H'$, and it sends $K$-rational points to $K$-rational points.
  \item If we fix a first order structure on $K$ expanding the field language, then $f$ sends $K$-definable subsets of $\mathbf H(K)$ to $K$-definable subsets of $\mathbf H'(K)$.\footnote{Beware that the similar-sounding terms ``defined over $K$'' and ``$K$-definable'' mean different things. However, if $\mathbf X \subset \mathbf K^n$ is defined over $K$, then $\mathbf X(K)$ is $K$-definable.}
  \item If $K$ is a valued field, then $f$ sends bounded subgroups of $\mathbf H$ to bounded subgroups of $\mathbf H'$; see Remark~\ref{boundedness}.
\end{itemize}

Several authors often speak about affine $K$-groups. 
Note that this is (up to isomorphism) the same as 
a linear algebraic $K$-group; see e.g. \cite[Proposition~1.10]{borel}.
(Recall that in this paper, all $K$-groups are assumed to be linear.)

The definitions of when $\mathbf H$ is a \emph{connected group}, a \emph{semisimple group} and a \emph{reductive group} can be found in \cite{borell, borel, springer, mar}.

We give some general terminology on almost simple groups.

\begin{definition} \label{def:simple}
We shall say that an abstract group $G$ is \emph{almost simple} if $G$ has no proper infinite normal subgroup. If $G$ is definable in some ambient first-order structure, we call it \emph{definably almost simple} if it has no proper infinite definable normal subgroup.
\end{definition}

\begin{remark} \label{rem:simple}
If $G$ is infinite and almost simple, and $H$ is a proper normal subgroup of $G$, then $H$ is finite, so $C_G(H)=G$ (as $|G:C_G(H)|$ is finite), that is, $H\leq Z(G)$. In particular, $G/Z(G)$ is (abstractly) simple. Likewise if $G$ is definably almost simple then $G/Z(G)$ is definably simple. 
\end{remark}

\begin{definition}\label{def:almost}
A $K$-group $\mathbf{H}$ is said to be \emph{almost $K$-simple} if it has no proper nontrivial connected normal $K$-subgroup. 
\end{definition}

Clearly, if $\mathbf{H}$ is almost simple as an abstract group, then it is almost $K$-simple, since a nontrivial connected $K$-subgroup is in particular infinite.

\begin{definition} \label{def:defg+}
Suppose $\mathbf{H}$ is a connected $K$-group. By $\mathbf{H}(K)^+$ we denote the normal subgroup of $\mathbf{H}(K)$ generated by the $K$-rational points of the unipotent radicals of parabolic $K$-subgroups of $\mathbf{H}$ \cite[1.5.2]{mar}, \cite[1.1]{tits}. 
\end{definition}

\begin{remark} \label{rem:g+}
When $K$  is perfect (as holds throughout this paper), $\mathbf{H}(K)^+$ can be characterized as the subgroup of $\mathbf{H}(K)$ generated by all unipotent elements  \cite[1.5.2]{mar}, \cite[Proposition 6.2$(i)$]{bortit}.
\end{remark}

\begin{remark} \label{rem:g+acf}
Note that $\mathbf{H}(K)^+$ is mostly of interest when $K$ is not algebraically closed. Indeed, if $K$ is algebraically closed and $\mathbf{H}$ is semisimple, then
$\mathbf{H}(K)^+ = \mathbf{H}(K)$ by \cite[27.5 (e)]{Humphreys}.
\end{remark}

\begin{definition} \label{def:isot}
Suppose $\mathbf{H}$ is a connected $K$-group and $K$ is a perfect field.
Then $\mathbf{H}$ is said to be \emph{isotropic over $K$} or \emph{$K$-isotropic} if $\mathbf{H}$ has an algebraic subgroup $\mathbf{T}$ defined over $K$ and isomorphic over $K$ to some torus,
\emph{i.e.}, to a group of the form $(\mathbf G_m)^t$ for some $t \ge 1$.
 Here $\mathbf{T}$ is what is called a \emph{$K$-split torus} (of $\mathbf H$).
We say that $\mathbf{H}$ is \emph{$K$-anisotropic} if it is not $K$-isotropic. See e.g. \cite[0.25]{mar}.
\end{definition}

Note that for $\mathbf{H}$ to be isotropic, we do not require the group $\mathbf{T}$ to be a maximal torus of $\mathbf{H}$. In particular, $\mathbf{H}(K)$ need not be a Chevalley group.

\begin{example}
The $\R$-subgroup $\SO_3 = \{A \in \SL_3 \mid AA^T=1\}$ of $\SL_3$ is $\R$-anisotropic, essentially since
$\SO_3(\R)$ is compact.
\end{example}

\begin{remark} \label{rem:semdense}
Suppose that $\mathbf{H}$ is a connected $K$-group.
\begin{enumerate}
\item[(i)] If $\mathbf{H}$ is semisimple, then $\mathbf{H}$ is $K$-isotropic if and only if $\mathbf{H}(K)^+$ is non-trivial \cite[1.5.2]{mar}. 
\item[(ii)] If $K$ is infinite and $\mathbf{H}$ is almost $K$-simple, $K$-isotropic and non-abelian, then $\mathbf{H}(K)^+$ is Zariski dense in $\mathbf{H}(K)$ \cite[1.5.4]{mar}, \cite[3.2(20)]{tits}. 
\item[(iii)] In the seminal paper \cite{tits}, Tits proved that if $K$ has at least 4 elements and $\mathbf{H}$ is semisimple and almost $K$-simple (Definition \ref{def:almost}), then $\mathbf{H}(K)^+$ is almost simple (see Definition~\ref{def:simple}). This theorem of Tits covers previous results of Dieudonn{\'e},  Chevalley, Steinberg and others on simplicity of classical linear algebraic groups of various special kinds, for example special linear groups $\SL_n(K)$ and symplectic groups $\Sp_{2n}(K)$. 
\item[(iv)] An example where $\mathbf{H}(K)^+$ is a Zariski dense proper subgroup of $\mathbf{H}(K)$ is given in \cite[1.3]{tits}. This is over a non-perfect field, and is also an example where $\mathbf{H}(K)^+$ is not the group generated by all unipotent elements of $\mathbf{H}(K)$.
\end{enumerate}
\end{remark}

\begin{definition}\label{def:simpconn}
A connected semisimple linear algebraic group $\mathbf{H}$ is called \emph{simply connected} if every central isogeny $\pi\colon \widetilde{\mathbf{H}} \to \mathbf{H}$, where $\widetilde{\mathbf{H}}$ is a connected linear algebraic group,  is an algebraic group isomorphism (see \cite[0.18, 1.4.9]{mar}, \cite[2.1.13]{plat-rap}, \cite[8.1.11]{springer}).
\end{definition}

\begin{remark}\label{rem:kneser}
\begin{enumerate}
\item[(i)] Every semisimple $K$-group $\mathbf{H}$ has a \emph{universal $K$-covering} $\pi \colon \widetilde{\mathbf{H}} \to \mathbf{H}$ defined over $K$ (\cite[Proposition 2.10]{plat-rap} \cite[2.6.1]{tits_class}); that is, $\widetilde{\mathbf{H}}$ is a simply connected $K$-group and $\pi$ is a central isogeny defined over $K$. (For more information about these notions, we refer the reader e.g. to \cite[Section 2.1.13]{plat-rap}.) Moreover, (cf. \cite[6.5]{bortit}) \[\pi\left(\widetilde{\mathbf{H}}(K)^+\right) = \mathbf{H}(K)^+.\]

\item[(ii)] The famous \emph{Kneser-Tits problem} asks whether $\mathbf{H}(K) = \mathbf{H}(K)^+$ holds if  $\mathbf{H}$ is simply connected, almost $K$-simple and isotropic over $K$ \cite{tits, gille}. It is known that, under these assumptions, $\mathbf{H}(K)=\mathbf{H}(K)^+$ holds, for example, if $K$ is algebraically closed
(see Remark~\ref{rem:g+acf})
or $K$ is a local field (e.g. $K=\Q_p$ the field of $p$-adic numbers) \cite[2.3.1(a)]{mar}, or if $K$ is a real closed field \cite[6.1, 6.8]{pe-pi-st}; however, in general $\mathbf{H}(K)^+$ might be a proper subgroup of $\mathbf{H}(K)$ -- Platonov gave a counterexample to the Kneser-Tits problem (see \cite{gille}, \cite{plat-rap} and \cite{tits_whit}).
\end{enumerate} 
\end{remark}

It would be very helpful for this paper to know conditions which ensure, for a semisimple almost $K$-simple linear algebraic $K$-group $\mathbf{H}$, that $\mathbf{H}(K)^+$ is definable in the field $(K,+,\times)$ (possibly equipped with a valuation). For example, if this holds then in Theorem~\ref{main} the conclusion could be strengthened to $G^*=\mathbf{H}(K)^+$. We include the following observation on this. 
\begin{lemma} \label{lem:kneser}
Let $\mathbf{H}$ be a semisimple, almost $K$-simple, linear algebraic $K$-group. Suppose that $\pi\colon\widetilde{\mathbf{H}}\to \mathbf{H}$ is the universal $K$-covering (see Remark \ref{rem:kneser}). Then
\begin{enumerate}
\item[(i)] if $\widetilde{\mathbf{H}}(K)^+=\widetilde{\mathbf{H}}(K)$, then $\mathbf{H}(K)^+$ is definable in $(K,+,\times)$;

\item[(ii)] if $K$ is  real closed or a local field, and $\mathbf{H}$ is $K$-isotropic, then 
$\mathbf{H}(K)^+$ is definable in $(K,+,\times)$.
\end{enumerate}
\end{lemma}

\begin{proof}
(i) We may assume that $\mathbf{H}$ is $K$-isotropic, as otherwise, by Remark~\ref{rem:semdense}(2) (i), $\mathbf{H}(K)^+$ is trivial, so clearly definable. By Remark~\ref{rem:kneser}, $\pi\left(\widetilde{\mathbf{H}}(K)\right)=\mathbf{H}(K)^+$.  Since $\widetilde{\mathbf{H}}$ and $\pi$ are defined over $K$, $\mathbf{H}(K)^+$ is definable in  the field $K$.

(ii) This follows immediately from (i) and 
Remark~\ref{rem:kneser} (ii). (Note that if $\mathbf H$ is almost $K$-simple, then so is $\widetilde{\mathbf H}$.)
\end{proof}

\subsection{Valued fields} \label{sec:val}

In most of the paper, $K$ will denote a field  with a non-trivial non-archimedean henselian valuation $v\colon K \to \Gamma\cup\{\infty\}$, where  $(\Gamma,<,+)$ is any totally ordered abelian group. If $K$ has positive characteristic $p$, then we usually assume that $K$ is algebraically closed. 
We implicitly always assume $v$ to be surjective.

As usual, the neutral element of $\Gamma$ is denoted by $0$, and we set $\gamma<\infty$ for all $\gamma\in\Gamma$.  We stress that the group $\Gamma$ may be non-archimedean, that is, may not embed in $({\mathbb R},<,+)$ -- we are dealing with {\em Krull} valuations.

We write $(K,v)$ for the valued field. To say that $(K,v)$ is \emph{Henselian} means that $v$ has  (up to equivalence) a unique extension to any algebraic extension of $K$. We shall denote by $K^{\alg}$ an algebraic closure of $K$ and abuse notation by using $v$ to also denote the unique valuation on $K^{\alg}$ extending $v$ on $K$.

\begin{notation}
If $(K,v)$ is a valued field, define $\O = \{x\in K : v(x)\geq 0\}$ (the \emph{valuation ring} of $(K,v)$), $\M = \{x\in K : v(x) > 0\}$ (its unique maximal ideal), $k = \O/\M$ (the \emph{ residue field}) and $\res\colon \O \to k$ (the \emph{residue map}).    
\end{notation}

\begin{notation}
For each $\gamma\in\Gamma$ and $a\in K^n$ define the \emph{open ball} $$B_{\gamma}(a)=\{x\in K^n : v(x_i-a_i)>\gamma \text{ for }i=1,\dots,n\}.$$    
\end{notation}

The family $\{B_{\gamma}(a) : \gamma\in\Gamma\}$ forms a basis of open neighbourhoods of $a$ in a topology, the \emph{valuation topology} on $K^n$.

We next introduce the (higher) leading term structures $\RV_\lambda$, a now-standard ingredient in the model theory of valued fields, and necessary for the next subsection.

\begin{definition}
Let $(K,v)$ be a valued field. Let $\lambda\in \Gamma$ with $\lambda\geq 0$, and define
$I_\lambda:=\{x\in K: v(x)>\lambda\}$, an ideal of $\mathcal{O}$. Let $\RV_\lambda^\times$ be the quotient of multiplicative groups $\RV_\lambda^\times=K^{\times}/(1+I_\lambda)$, let
$\RV_\lambda:=\RV_\lambda^{\times}\cup\{0\}$, and let
$\rv_\lambda:K\to \RV_\lambda$ be the map extending the projection $K^\times \to \RV_\lambda^{\times}$, with $\rv_\lambda(0)=0$.
We write $\RV$ and $\rv$ for $\RV_0$ and $\rv_0$ respectively.
\end{definition}

To get some intuition about $\RV$,
note that since the subgroup $\mathcal{O}^\times/(1+\mathcal{M})$ of $K^\times/(1+\mathcal{M})$ is isomorphic to the multiplicative group $k^\times$, we have a short exact sequence
$$k^\times  \rightarrow {\rm RV}^\times \to \Gamma.$$
Another point of view is that we may identify $\RV^\times$ with the set of open balls $\{B_{v(a)}(a):a \in K^\times\}$.

\subsection{Model theory and 1-$h$-minimality} \label{sec:hen-min}

We will consider our valued field $(K, v)$ as a structure in a suitable first-order language.
Many of our results hold for various different languages, the most basic one being Abraham Robinson's one-sorted language
$L_{\divv}=(+,-,\cdot,0,1,\divv)$, where $\divv$ is a binary relation symbol with $K\models \divv(a,b)$ if and only if $v(a)\leq v(b)$.
It follows from \cite{rob} that any algebraically closed valued field has quantifier elimination in this language, and conversely, by \cite{MMD}, any non-trivially valued field with quantifier elimination in $L_{\divv}$ is algebraically closed.

Most of the time, we will either assume $(K,v)$  
to be a model of $\ACVF_p$, \emph{i.e.} an algebraically closed valued field of positive characteristic $p$, in the language $L_{\divv}$, or we will assume that $\chr(K) = 0$ and we work in any 1-$h$-minimal expansion $(K, v, \dots)$, a notion which we recall below.

The value group, the residue field and the leading term structure $\RV$ will be considered as imaginary sorts. Also, we can consider the disjoint union
$\coprod_{\lambda \ge 0}\RV_\lambda$ of all higher leading term structures as an imaginary sort.

The notion of 1-$h$-minimality has been introduced in equi-characteristic 0 in \cite{iCR.hmin} and in mixed characteristic in \cite{hmin2}.  It provides the framework of Theorem~\ref{main}. We discuss key examples, and describe the consequences of 1-$h$-minimality that we need. Note that we work with the above notation for valuations, not the multiplicative notation from \cite{iCR.hmin} and \cite{hmin2}.

\begin{definition}
Let $(K,v,\ldots)$ be a valued field, and let $\lambda\in \Gamma$ with $\lambda>0$. If $C\subset K$ is finite, and $X\subset K$ is arbitrary, we say that $C$ {\em $\lambda$-prepares} $X$ if the following holds: for any $x,x'\in K$ such that $\rv_\lambda(x-c)=\rv_\lambda(x'-c)$ for all $c\in C$, we have 
$$x\in X \Leftrightarrow x'\in X.$$
\end{definition}

In \cite{iCR.hmin} the authors stress an analogy with an o-minimal field $R$: given a definable $X\subset R$, there is a finite set $C$ such that for $x\in R$, whether or not $x\in X$ depends on the tuple
$({\rm sgn}(x-c):c\in C)$; we write $X$ canonically as a finite union of open intervals and singletons, and $C$ consists of the singletons and endpoints of the intervals.
In $1$-h-minimality, $\rm sgn(x-c)$ is replaced by $\rv_\lambda(x-c)$ (for certain $\lambda$).
In the o-minimal case, $C$ is definable from a canonical parameter for $X$ (in particular from any parameters used to define $X$) but this issue is more subtle in the valued field case and is the point of the next definition.

\begin{definition}\label{1h}
Let $T$ be a theory of valued fields of characteristic 0 in a language $L\supseteq L_{{\rm div}}$. We say that
$T$ is {\em 1-$h$-minimal} if the following holds:
let $K\models T$, let $n\in \mathbb{N}^{>0}$,  let $\lambda\in \Gamma$ with $\lambda>0$, let $A\subset K$ and let $\xi \in \RV_\lambda$, a single element. Then for any $(A \cup \RV_{v(n)}\cup \{\xi\})$-definable set $X\subset K$, there is $m\in \mathbb{N}^{>0}$ such that $X$ is $(\lambda +v(m))$-prepared by an $A$-definable finite
set $C\subset K$. One calls a structure $1$-h-minimal if its theory is.
\end{definition}

This definition, from 
\cite[Definition 2.2.1]{hmin2} and applicable in the mixed characteristic case, extends that given in \cite{iCR.hmin} for equicharacteristic 0 -- in the latter case we have $v(n) = v(m)=0$ so $n$ and $m$ play no role. In fact several versions of this definition are given in \cite[Section 2.2]{hmin2}, and they are shown to be equivalent in Theorem 2.2.8. One version states that for any $(K,v,\ldots)\models T$, if $v'$ is the finest equicharacteristic 0 coarsening of $v$, then $(K,v',\ldots)$ is 1-$h$-minimal in the sense of \cite{iCR.hmin}. Proofs from 1-$h$-minimality in mixed characteristic typically reduce to proofs in equicharacteristic 0 using $\RV$-resplendency results (\cite[Section 2.6]{hmin2}) and the definition of 1-$h$-minimality via the finest equicharacteristic 0 coarsening; see e.g. \cite[Theorem 2.2.8]{hmin2}.

The notion of
1-$h$-minimality is one of a hierarchy of conditions $\ell$-$h$-minimality for $\ell\in \mathbb{N}\cup\{\omega\}$, where $\ell$ governs the number of elements $\xi$ from $\RV_\lambda$ in Definition~\ref{1h}. These have
increasing strength as $\ell$ increases, and it  is not known whether they are all equivalent. These are referred to generically as {\em Hensel minimality}. Provided we assume finite ramification\footnote{meaning that the valuation takes finitely many values between 0 and $v(p)$} in the mixed characteristic case (or allow infinite ramification but add `definable spherical completeness' as an assumption), they all imply that the models of $T$ are Henselian as valued fields
(in equicharacteristic 0, see \cite[Theorem 2.7.2]{iCR.hmin}, and for mixed characteristic, see \cite[Remark 2.2.3]{hmin2}).

\begin{example}  \label{ex1h} We list several types of examples of 1-$h$-minimal expansions of valued fields. For more detail concerning the examples, see \cite{iCR.hmin}.

 1. Henselian valued fields of equicharacteristic 0 (\cite[Theorem 6.2.1]{iCR.hmin}), or mixed characteristic (\cite[Corollary 6.2.7]{iCR.hmin}), possibly equipped with analytic structure in the sense of \cite{CL.analyt}, are 1-$h$-minimal. These fields with analytic structure form a broad family stemming ultimately from the construction of ${\mathbb Q}_p^{{\rm an}}$ in \cite{denef}. (The structure ${\mathbb Q}_p^{{\rm an}}$ is essentially an expansion of ${\mathbb Q}_p$ by multivariable functions of the form 
$f(x)=\Sigma_{\mu \in {\mathbb N}^k} a_\mu x^\mu$, where 
$x=(x_1,\ldots,x_k)$ and $a_\mu\in {\mathbb Z}_p$, with 
$v(a_\mu) \to \infty$ as $|\mu| \to \infty$, where $|\mu|=\mu_1+\ldots+\mu_k$; such functions converge on $\mathbb{Z}_p^k$, and are defined to take value 0 elsewhere.)

2. Let $T_0$ (in a language $L_0$) be an o-minimal expansion of the theory of real closed fields, which is {\em power bounded} in the sense that for $K\models T_0$, and any $L_0(K)$-definable function $f:K \to K$, there is a `power function' $g:K^\times \to K^\times$ (a definable endomorphism of $(K^\times,\cdot)$) such that $|f(x)|\leq g(x)$ for all sufficiently large $x$. Let $K_0\prec K$ be models of $T_0$, let $\mathcal{O}_K$ be the convex closure of $K_0$ in $K$ and assume that $\mathcal{O}_K$ is a proper subring of $K$. Then $\mathcal{O}_K$ is a valuation ring of $K$ and the expansion of the $L_0$-structure $K$ to a language $L$ with a predicate for $\mathcal{O}_K$ is referred to as a $T_0$-convex valued field. (Of course, the valuation can also be defined using the binary predicate ${\rm div}$ in place of the one defining $\mathcal{O}_K$.) By \cite[Theorem 6.3.4]{iCR.hmin}, such a structure is 1-$h$-minimal. 

3. Any expansion of a 1-$h$-minimal field by predicates interpreted by subsets of Cartesian powers of $\RV$ is 1-$h$-minimal
-- see \cite[Section 1.2]{iCR.hmin} in the equicharacteristic $0$ case and \cite[Proposition~2.6.5 (ii)]{hmin2} in mixed characteristic. For example, we may expand the value group $\Gamma$ by a predicate for a proper convex subgroup $\Gamma_0$, ensuring that the coarsening valuation $K \to \Gamma \to \Gamma/\Gamma_0$ is definable. If the coarsened valuation is of equicharacteristic 0,\footnote{Maybe this also works if the coarsened valuation is of mixed characteristic, but the proof has not been written down.} 
then by \cite[Theorem 2.2.8 (4)$\Rightarrow$ (2)]{hmin2}, in this expanded language, $K$ remains $1$-h-minimal if we consider it as a valued field with the coarsened valuation.

 \end{example}

We next recall how one defines dimension and differentiation in the valued field context. 

\begin{definition}\label{def:dim}
The (topological) \emph{dimension} $\dim X$ of a definable set $X \subset K^n$ is the maximal $d$ such that there exists a coordinate projection $\pi\colon K^n \to K^d$ such that $\pi(X)$ contains a ball.
(By a \emph{coordinate projection}, we mean a map $(x_1, \dots, x_n) \mapsto (x_{i_1}, \dots, x_{i_d})$, for some $1 \le i_1 < \dots < i_d \le n$.) We set $\dim X := -\infty$ if $X$ is empty.
\end{definition}
A different definition of dimension is given (in equicharacteristic 0) in \cite[Definition 5.3.1]{iCR.hmin}, but  (under 
1-$h$-minimality) this is shown to be equivalent to the above in Proposition 5.3.4(1) (see also \cite[Proposition 3.1.1(3) part (1)]{hmin2} for mixed characteristic). 

In characteristic zero, differentiation is defined in the usual way, using limits with respect to the valuation topology. In the following definition, we define the valuation of a tuple to be the minimum of the valuations of its entries: For $x = (x_1, \dots, x_n) \in K^n$, set $v(x) := {\rm min}_i v(x_i)$.

We shall work with the following notion of {\em strict differentiability}, adopting the presentation of \cite{AH}.

\begin{definition}\label{d.deriv}
Assume $(K,v)$ is a valued field of characteristic zero. Let $U \subset K^m$ be an open set and $f\colon U \to K^n$ be a map.
We say that $f$ is {\em strictly differentiable} at $a\in U$ if there is a linear map $A:K^m\to K^n$ such that  for every $\epsilon\in \Gamma$ there is $\delta\in \Gamma$ such that for all $x,y\in K^m$ with $v(x-a)>\delta$ and $v(y-a)>\delta$, we have 
$$v(f(x)-f(y)-A(x-y))\geq v(x-y)+\epsilon.$$
We say $f$ is {\em strictly differentiable on $U$} if it is strictly differentiable at every point of $U$. 
\end{definition}

As noted in \cite{AH}, the map $A$ above is uniquely determined, so may be denoted $d_af$. If $f$ is strictly differentiable, then it is continuously differentiable in the more usual sense.

\newcommand{\bad}{{\mathrm{bad}}}

The following proposition collects some key consequences of 1-$h$-minimality. Note, for here and later,  that part of the definition of 1-$h$-minimality includes that the field has characteristic 0 (though the residue field may have prime characteristic).

\begin{proposition} \label{oldhyp} 
Assume that $K$ is either an algebraically closed valued field of characteristic $p$, or is a 1-$h$-minimal expansion of a henselian valued field. Then the following hold.
\begin{enumerate}
\item[(i)] Every infinite definable subset of $K$ contains a ball. Hence 
any non-empty definable set $X \subset K^n$ has dimension $0$ if and only if it is finite.
\item[(ii)] If $f\colon X \to Y$ is a definable map (for some definable $X \subset K^m, Y \subset K^n$) and if
       there exists a $d \ge 0$ such that $\dim f^{-1}(y) = d$ for every $y \in Y$, then $\dim X = \dim Y + d$.
       In particular,  if there exists a definable bijection between some definable sets $X \subset K^m$ and $Y \subset K^n$, then $\dim X = \dim Y$,
and for definable sets $X \subset K^m$ and $Y \subset K^n$, we have $\dim (X \times Y) = \dim X + \dim Y$.
\item[(iii)] For definable sets $X, X' \subset K^n$, we have  $\dim (X \cup X') = \max\{\dim X, \dim X'\}$. In particular, $X \subset X'$ implies $\dim X \le \dim X'$.
\item[(iv)]  For every non-empty definable set $X$, the difference $\bar X \setminus X$ has strictly lower dimension than $X$, where $\bar X$ is the closure of $X$ with respect to the valuation topology.
 \end{enumerate}
 \end{proposition}

 \begin{proof}
 \begin{enumerate}
\item[(i)] If $K$ is algebraically closed of characteristic $p$, this follows directly from quantifier elimination. In the 1-h-minimal case it follows almost immediately from the definition. The second assertion follows immediately. 
\item[(ii), (iii), (iv)] For $\ACVF_p$, see e.g. \cite[2.24(4), Proposition 2.15]{lou}. In the 1-$h$-minimal context, see \cite[Proposition 5.3.4]{iCR.hmin} in the equicharacteristic zero case, and \cite[Proposition 3.1.1(3) parts (4)--(6)]{hmin2} in the mixed characteristic case.
\end{enumerate}
 \end{proof}
\begin{remark} \label{strict-diff-rem}
If $K$ has characteristic 0, then for every definable function $f\colon K \to K$, there exists a finite subset $A \subset K$ such that $f$ is strictly differentiable on $K \setminus A$. This follows from \cite[Proposition 3.13]{AH}, in combination with Proposition~\ref{oldhyp}(i) above.
\end{remark}

\begin{lemma} \label{interior}
    Under the assumptions of Proposition~\ref{oldhyp}, suppose that $U\subset V$ are definable sets with $\dim(U)=\dim(V)$. Then $U$ has non-empty interior in $V$.
\end{lemma}
\begin{proof} The interior of $U$ in $V$ can be written as $U\setminus Y$, where $Y=\bar{Z}\setminus Z$ and $Z=V\setminus U$. By Proposition~\ref{oldhyp}(iv), $\dim(Y)<\dim(Z)\leq \dim(V)=\dim(U)$, so $U\setminus Y$ is non-empty.
\end{proof}

The following lemma states that the dimension of a definable set $X$ is equal to the maximal ``local dimension'' at points of $X$. Note that clearly, local dimensions cannot exceed the global one, so the only question is whether the global dimension is realized locally.

\begin{lemma}\label{l.locdim}
Let $K$ be a model of $\ACVF_p$ or a 1-$h$-minimal expansion of a valued field. 
Let $X \subset K^n$ be a definable set. Then there exists an $x \in X$ such that for every ball $B \subset K^n$ containing $x$, $\dim (X \cap B) = \dim X$.
\end{lemma}

\begin{proof} For equicharacteristic 0, see \cite[Proposition 5.3.4(5)]{iCR.hmin}, and for mixed characteristic it follows from \cite[Proposition 3.1.1(3) part (5)]{hmin2}.

In the case when $K$ is a pure algebraically closed valued field of characteristic $p$, the result follows immediately from Theorem 3.1 of \cite{iF.dim}; the assumptions of that theorem are very similar to Dim(1)-Dim(4) of \cite{lou}, and easily verified for $\ACVF_p$.
\end{proof}

The following result shows that for $L_{{\rm div}}$-definable sets, topological dimension behaves well with respect to Zariski closure and Zariski dimension.

\begin{lemma}\label{lou}
Suppose that $X\subset K^n$ is $L_{\divv}$-definable with parameters.
Recall that we write $\cl_K(X)$ for the Zariski closure of $X$ in $K^n$,
and let us denote by $\mathbf Y = \cl_{K^{\alg}}(X)$ the the Zariski closure of $X$ in $(K^{\alg})^n$.  Then
\begin{enumerate}
\item[(i)] $X$ has non-empty interior (with respect to the valuation topology) in $\cl_K(X)$.
\item[(ii)]
$\dim X=\dim \cl_K(X) = \zardim(\mathbf Y)$
(where $\zardim$ denotes the Zariski dimension; see Section~\ref{sec:lin}).
\end{enumerate}
\end{lemma}
Recall in the above that from Remark~\ref{note-on-dim},  $\mathbf Y$ is a variety defined over $K$ and  $\mathbf Y(K) = \cl_K(X)$.

\begin{proof}[Proof of Lemma~\ref{lou}]
(i) See \cite{lou}; namely Proposition 2.18, in combination with 
 2.24 (where Example (4) handles $\ACVF_p$) and  Theorems 3.5 (for equicharacteristic 0) and  3.9 (for mixed characteristic).
 
 (ii)
 It is easily seen that topological dimension as defined above is a {\em dimension function} on the $L_{\divv}$-definable sets, in the sense of \cite[Dim(1)--Dim(4)]{lou}. Conditions Dim(1) and Dim(3) are immediate from the definition of topological dimension, Dim(2) follows from Proposition~\ref{oldhyp}(iii), and Dim(4) is a consequence of  Proposition~\ref{oldhyp}(ii).

The algebraic dimension of a subset $X \subset K^n$ introduced \cite[Definition 2.2]{lou} is just the Zariski dimension $\zardim(\cl_{K^{\alg}}(X))$ of its Zariski closure,
and the assertions of \cite[Theorem 3.5, 3.9]{lou} are that algebraic dimension is the {\em only} dimension function on the $L_{\divv}$-definable sets in henselian valued fields of characteristic 0. The uniqueness proof in Theorem 3.5 applies also for $\ACVF_p$. It follows that the two notions of dimension coincide and hence that both $\dim X$ and $\dim \cl_K(X)$ are equal to $\zardim \mathbf Y$.
\end{proof}

\section{Manifolds and Lie Groups over valued fields}

Throughout this section, we shall work in a 1-$h$-minimal expansion $(K,v,\ldots)$ of a valued field of characteristic 0. In Subsection~\ref{sec:mani} we summarise some material from \cite{AH} which ensures that our definable group $G$ can be endowed with a definable differential manifold structure in an appropriate sense. Subsection~\ref{sec:tan-Lie} applies this, along with a line of argument in the o-minimal case from \cite{pps}, to obtain some results on the Lie algebra structure on the tangent space.

\subsection{Manifolds and Lie groups}
\label{sec:mani}

We need a notion of definable manifold over $K$. Several such notions have been provided in \cite{AH}. The one we will be using is what the authors of \cite{AH}  call a \emph{definable weak strictly differentiable manifold}; we will abbreviate this by \emph{dwsd manifold}. Mostly, the definition of such manifolds is what one would expect: it is a definable set $M$ equipped with definable charts $U_i \to M$ satisfying usual conditions on transition maps. However, the adjective ``weak'' expresses that a chart is not exactly what one might think: Instead of requiring the sets $U_i$ to be subsets of $K^n$ (for $n = \dim M$), one needs to allow them to be some kind of finite (étale) covers of such sets.

A concrete example where such finite covers are necessary is
the ``circle'' $M = \{(x,y) \in K^2 \mid x^2 + y^2 = 1\}$, which clearly should be a manifold. We would like to consider $U := M \setminus \{(\pm 1, 0)\}$ as the disjoint union of the upper and lower half circle, but ``upper'' and ``lower'' make no sense, and there is no way to split $U$ \emph{definably} into two branches. Instead, the notion of chart is modified in such a way that the entire $U$ is a chart.

So here is the definition of dwsd manifolds (\cite[Definitions~5.1--5.3]{AH}):

\begin{definition}\label{d.mnf}
\begin{enumerate}
  \item
  A definable map $r\colon U \to K^n$ (where $U \subset K^m$ for some $m$) is called \emph{(topologically) étale} if every $u \in U$ has a neighbourhood $B \subset U$
  such that the restriction $r|_{B}$ is a homeomorphism onto its image $r(B)$ (with respect to the valuation topology on $U$ and on $K^n$).
  \item
  Given two definable étale maps $r_i\colon U_i \to K^{n_i}$ ($i = 1,2$), we call a definable function $f\colon U_1 \to U_2$ \emph{strictly differentiable}, if every $y \in U_1$ has a neighbourhood $B \subset U_1$ such that $r_1|_B$ is injective and $r_2 \circ f \circ (r_1|_{B})^{-1}$ is strictly differentiable on $r_1(B)$ (in the sense of Definition~\ref{d.deriv}).
    \item
    A \emph{definable $n$-chart} on a definable set $M$ consists 
    of a definable injection $\phi\colon U \hookrightarrow M$ and
    an étale definable map $r\colon U \to K^n$ (where $U \subset K^m$ for some $m$). We often call ``$\phi\colon U \hookrightarrow M$'' the chart, having the map $r$ implicitly in mind.
    \item
    A  \emph{definable weak strictly differentiable $n$-manifold}, or \emph{dwsd $n$-manifold} for short, consists of a definable set $M$ together with finitely many definable $n$-charts. $\phi_i\colon U_i \hookrightarrow M$ such that
    \begin{enumerate}
        \item $M$ is covered by the images $\phi_i(U_i)$;
        \item For each pair $i,j$, the set $U_{ij} := \phi_i^{-1}(\phi_j(U_j))$ is open in $U_i$; and
        \item the transition maps $U_{ij} \to U_{ji}, \phi_j^{-1}\circ \phi_i$ are strictly differentiable in the sense of (2).
    \end{enumerate}
    We endow $M$ with the unique topology turning all maps $\phi_i$ into open immersions (which is possible by (b)). In the following, when we refer to charts of a manifold $M$, we always mean one of those $\phi_i$.

    A \emph{dwsd manifold} is a dwsd $n$-manifold for some $n$.
    \item
       Given two dwsd manifolds $M, N$, a definable map $f\colon M \to N$ is called \emph{strictly differentiable} if  for all charts $\phi_i\colon U_i \to M$ and $\psi_j\colon V_j \to N$,
       the set $W_{ij} := \phi_i^{-1}(f^{-1}(\psi_j(V_j)))$ is open in $U_i$ and
       the composition $\psi_j^{-1} \circ f\circ \phi_i|_{W_{ij}}$ is strictly differentiable on $W_{ij}$ (in the sense of (2)). Dwsd manifolds are considered as a category using strictly differentiable definable maps as morphisms.
\end{enumerate}
\end{definition}

\begin{remark}\label{rem:manichoices}
In the above definition, various choices have been made; we follow \cite{AH} concerning those choices; in particular:
    \begin{enumerate}
        \item The topology on the $U_i$ is always the one induced from the ambient space; in contrast, its differentiable structure is induced by $r_i$.
        \item On $M$, both the topology and the differentiable structure are the ones induced by the $\phi_i$.
        \item It is imposed that a manifold $M$ has a cover by \emph{finitely many} charts. This implies that if $M$ is an $n$-manifold, then $n$ is also equal to the dimension of $M$ as a definable set (in the sense of Definition~\ref{def:dim}). 
    \end{enumerate}
\end{remark}

\begin{remark}
Formally, a dwsd manifold comes with a choice of finitely many charts.
However, replacing those charts by other compatible ones yields an isomorphic manifold and we care about manifolds only up to isomorphism. 
\end{remark}

\begin{remark}
In Definition~\ref{d.mnf}, one could replace ``strictly differentiable'' everywhere by some other condition. Most importantly,
to obtain that every Lie group gives rise to a  Lie algebra (\cite[Proposition~6.19]{AH}), \cite{AH} use
$T_2$-manifolds. Here, $T_2$ is a rather technical condition about maps which morally is something like ``strictly differentiable of order two'', but which does not follow from twice strictly differentiable. Since we will use \cite[Proposition 6.19]{AH}) as a black box, we do not introduce this notion here.
\end{remark}
\begin{remark} We justify the assertion above that the $U_i$ are finite covers. 
Consider $r\colon U \to K^n$ as in \ref{d.mnf}(1) above and pick any $a$ in the range of $r$.  Then $r^{-1}(a)$ is finite, for if it were infinite, then by 1-$h$-minimality it would not be discrete; we could then choose non-isolated $u\in r^{-1}(a)$, and
 the \'etaleness condition fails for this $u$. 
    
\end{remark}

Next, we recall the notion of tangent space from \cite{AH}. The definition might look a bit odd at first sight (in particular Definition~\ref{d.diff} (1)), but it comes out naturally from the notions of differentiable maps given in Definition~\ref{d.mnf} (2), (5).

\begin{definition}[{\cite[Definition~5.4]{AH}}]\label{d.diff}
\begin{enumerate}
\item 
Given a definable étale map $r_1\colon U_1 \to K^{n_1}$ and a point $y \in U_1$, the \emph{tangent space} $T_yU_1$ is simply defined to be $K^{n_1}$. 

If $r_2\colon U_2 \to K^{n_2}$ is a second definable étale map and $f\colon U_1 \to U_2$ is strictly differentiable, the \emph{differential} $d_yf\colon T_yU_1 \to T_{f(y)}U_2$ is defined to be the strict derivative of the map $r_2 \circ f \circ (r_1|_{B})^{-1}$ from Definition~\ref{d.mnf} (2).
\item
Given a dwsd manifold $M$ with charts $\phi_i\colon U_i \to M$, and given a point $x \in M$, the \emph{tangent space} $T_xM$ is obtained as follows:
Let $I$ be the set of indices $i$ such that $x$ lies in $\phi_i(U_i)$,
set $y_i := \phi_i^{-1}(x)$, take the disjoint union
$\coprod_{i \in I} T_{y_i}U_i$
of the corresponding tangent spaces, and identify two points $v \in T_{y_i}U_i$ and $w \in T_{y_j}U_j$ if the 
differential
$d_{y_j}(\phi_i^{-1}\circ \phi_j)$
of the corresponding transition map sends $w$ to $v$. Note that for every $i$, we obtain a natural identification of $T_xM$ with $T_{y_i}U_i$.
\item
If $N$ is a second dwsd manifold, $f\colon M \to N$ is strictly differentiable, and $x\in M$, then the \emph{differential} $d_xf\colon T_xM \to T_{f(x)}N$ is obtained by taking any charts $\phi_i\colon U_i \to M$ and $\psi_j\colon V_j \to N$ having $x$ and $f(x)$ in their images, respectively, and defining $d_xf$ to be the map induced by
$d_{\phi_i^{-1}(x)}(\phi_j^{-1} \circ f\circ \phi_i)\colon 
T_{\phi_i^{-1}(x)}U_i \to T_{\phi_j^{-1}(f(x))}V_j$.
\end{enumerate}
\end{definition}

\begin{definition}
A \emph{dwsd Lie group} is a definable group $G$ equipped with a dwsd manifold structure, such that the group multiplication $G\times G\to G$ and the group inversion $G\to G$ are both morphisms of dwsd manifolds (\emph{i.e.}, are strictly differentiable). A \emph{homomorphism of dwsd Lie groups} is a group homomorphism which is  also a morphism of dwsd manifolds (\emph{i.e.}, it is definable and strictly differentiable).
\end{definition}

We can now state one of the central results of \cite{AH}, which asserts that definable groups are Lie groups. (Recall that the only non-imaginary sort of our structure $(K,v,\dots)$ is $K$ itself, so the proposition does not apply e.g. to the value group.)

\begin{proposition}[{\cite[Proposition~6.4]{AH}}]\label{p.Gman}
Any definable group can be equipped with the structure of a dwsd manifold in such a way that it becomes a dwsd Lie group.
\end{proposition}

As noted after  the proof of \cite[Proposition 6.4]{AH}, the dwsd Lie group structure on a definable group $G$ is unique up to a unique isomorphism. In the rest of this section, when referring to a definable group $G$, any implicit references to a manifold structure on $G$ refer to this dwsd manifold structure.

\subsection{Tangent spaces and Lie algebras} \label{sec:tan-Lie}

In this subsection we overview some fairly standard material on tangent spaces and Lie algebras, but in the definable non-archimedean 1-$h$-minimal context, using the framework from \cite{AH} described in the previous subsection. The goal of the subsection is Theorem~\ref{simpleliealg} below.
Our treatment is influenced by \cite{pps}, which itself is an adaptation of the exposition in \cite{vo} for the context in \cite{pps} of groups definable in o-minimal structures.
Any definable group which we mention is assumed to be equipped with the unique dwsd manifold structure obtained using Proposition~\ref{p.Gman}.
(In particular, the topology on the group is not necessarily the one from the ambient $K^n$.)
We denote the identity element of a group by $e$.

\begin{lemma} \label{opensubgroup}
Let $H \leq G$ be definable groups. Then
\begin{enumerate}
\item[(i)] $H$ is a closed subset of $G$ (with respect to the topology on $G$ given by the dwsd manifold structure on $G$), and furthermore, the inclusion $i:H \to G$ is a closed embedding of dwsd manifolds.
\item[(ii)] $H$ is open in $G$ if and only if $T_e(H)=T_e(G)$ if and only if $\dim(H)=\dim(G)$. Those conditions in particular hold if $H$ has finite index in $G$.
\end{enumerate}

\end{lemma}
\begin{proof} 
Part (i) and the first equivalence in (ii) follow from Proposition 6.9 (with Proposition 6.10) and Corollary 6.14 respectively of \cite{AH}. For the second equivalence in (ii), note that by Remark~\ref{rem:manichoices}(3), we have $\dim G = \dim T_e(G)$, and similarly for $H$.
Finally, if $H$ has finite index in $G$, then its complement $G \setminus H$ is a finite union of translates of $H$ and hence closed in $G$, so $H$ is open in $G$.
\end{proof}

If $H_1$, $H_2$ are subgroups of a group $G$, then we say that $H_1,H_2$ are \emph{commensurable},
if $H_1 \cap H_2$ has finite index in each of $H_1$ and $H_2$.

\begin{lemma}\cite[Claim 2.20]{pps} \label{2.20}
Let $H_1,H_2$ be definable subgroups of a definable group $G$
which are commensurable.
Then $T_e(H_1)=T_e(H_2)$. 
\end{lemma}

\begin{proof}
By Proposition~\ref{oldhyp}(iii) we have $\dim(H_1\cap H_2)=\dim(H_i)$ for each $i$. Hence, by Lemma~\ref{opensubgroup}, $H_1\cap H_2$ is open in $H_i$, so $T_e(H_1 \cap H_2)=T_e(H_i)$ (for $i=1,2$).
\end{proof}

For the next result, \cite[Corollary 6.11]{AH} provides an important tool. In the o-minimal setting, this is \cite[Theorem 2.21]{pps}.

\begin{proposition}\label{2.21} Let $G,H$ be definable groups and $\phi\colon G \to H$ a homomorphism of dwsd Lie groups (\emph{i.e.} a group homomorphism which is definable and strictly differentiable). Let $H_1\leq H$ be definable, and $G_1:=\phi^{-1}(H_1)$. Then $T_e(G_1)=(\dd_e(\phi))^{-1}(T_e(H_1))$.
\end{proposition}

\begin{proof}
We clearly have the inclusion $\subset$, so it suffices to show that both sides have the same dimension.

Consider the maps
\[
f\colon G_1 \to G \times H_1,
a \mapsto (a, \phi(a))
\]
and
\[
g_1,g_2\colon G \times H_1 \to H,
g_1(a,b) = \phi(a), g_2(a,b) = b.
\]
One easily verifies that $f$ is the equalizer of $g_1$ and $g_2$, \emph{i.e.}, $f$ is injective and its image is exactly the set $Z$ of points where $g_1$ and $g_2$ agree.
In a similar way, one verifies that
the equalizer of $\dd_e g_1, \dd_e g_2\colon T_e(G) \times T_e(H_1) \to T_e(H)$ is the map
\[
(\dd_e(\phi))^{-1}(T_e(H_1))
\to  T_e(G) \times T_e(H_1),
w \mapsto (w, \dd_e(\phi)(w)).
\]

Applying \cite[6.11]{AH} and putting everything together, we obtain
\[
\dim T_e(G_1) = \dim T_e(Z) = \dim Z = \dim (\Ker(\dd_e g_1 - \dd_e g_2))
= \dim (\dd_e(\phi))^{-1}(T_e(H_1)).
\]
\end{proof}

As in the o-minimal setting (see \cite[Theorem 2.24]{pps}), we deduce the following:

\begin{corollary}\label{2.24}
Let $G$ be a definable group, $V$ an $n$-dimensional definable vector space over $K$, and $\phi:G \to \Aut(V)$ a definable homomorphism. Let $U$ be a subspace of $V$, and put $G_1:=\{g\in G: \phi(g)(U)\subseteq U\}$ and $G_2:=\{g\in G: \phi(g)|_U=\id|_U\}$.
Then, viewing  $\dd_e(\phi)$ as a linear function from $T_e(G)$ to $\End(V)$ (the latter identified with $T_e({\rm Aut}(V))$), we have
\begin{enumerate}
    \item[(i)] $T_e(G_1)=\{\xi\in T_e(G):\dd_e(\phi)(\xi)(U)\subseteq U\}$
    \item[(ii)] $T_e(G_2)=\{\xi\in T_e(G):\dd_e(\phi)(\xi)(U)=\{0\}\}$.
\end{enumerate}
\end{corollary}

\begin{proof}
\begin{enumerate}
\item[(i)] This follows from Proposition~\ref{2.21}, putting $H=\Aut(V)$ and $$H_1:=\{x\in \Aut(V):x(U)\leq U\}.$$
\item[(ii)] Again, apply Proposition~\ref{2.21}, putting 
$$H_2:=\{x\in \Aut(V): x|_U=\id|_U\}=\{x\in \Aut(V): (x-\id)(U)=\{0\}\},$$ and noting that
$$T_e(H_2)=\{x\in \End(V): x(U)=\{0\}\}.$$
\end{enumerate}
\end{proof}

\begin{corollary} \label{2.22}
Let $G$ be a definable group, and $f$ a definable strictly differentiable Lie group automorphism of $G$. Then $\dd_e(f) \in\Aut(T_e(G))$.
\end{corollary}
\begin{proof} See \cite[Corollary 2.22]{pps}. Apply Proposition~\ref{2.21} with $H=G$ and with $H_1$ as the trivial group, to obtain that $\dd_e(f)$ has trivial kernel and so  lies in $\Aut(T_e(G))$.
\end{proof}

Let $G$ be a definable group. We recall the Lie algebra structure on $T_e(G)$ obtained in \cite[6.15--6.19]{AH}; see \cite[Section 2.4]{pps} for the o-minimal analogue. 
For $g\in G$,  we write $\Int_g$ for the map $G \to G$ given by $\Int_g(x)=gxg^{-1}$.  Put $\Ad g:=\dd_e(\Int_g)$. Then  by
Corollary~\ref{2.22}, $\Ad g \in \Aut(T_e(G))$. Also,  $\Ad\colon G \to \Aut(T_e(G))$ is a strictly differentiable definable homomorphism, below called the {\em adjoint} representation of $G$. 

Now let $\adj:= \dd_e(\Ad)$. Then $\adj\colon T_e(G)\to \End(T_e(G))$ is a definable homomorphism. Define the operation
$[-,-]$ on $T_e(G)$ by $[\xi,\zeta]=\adj(\xi)(\zeta)$.
\begin{lemma}\label{2.27}
$(T_e(G), [-,-])$ is a Lie algebra.
\end{lemma}
\begin{proof}
See \cite[Proposition 6.19]{AH}.
\end{proof}

We denote by $L(G)$ the Lie algebra of $G$ as given by the last lemma.

\begin{remark}
A standard computation shows that for $g \in G$, $\Ad g \in \Aut(T_e(G))$ is even a Lie algebra automorphism: From $\Ad(ghg^{-1}) = (\Ad g) \circ (\Ad h) \circ (\Ad g)^{-1}$, one obtains $\Ad \circ \Int_g = \Int_{\Ad g} \circ \Ad$
(where $\Int_{\Ad g}\colon \Aut(T_e(G)) \to \Aut(T_e(G)), \phi \mapsto (\Ad g) \circ \phi \circ (\Ad g)^{-1}$).
We take the derivative at $e$ on both sides, and use that since
$\Int_{\Ad g}$ is a linear automorphism of the vector space $V=\End(T_e(G))$, it is its own derivative -- indeed, for any $v\in V$ the tangent space to $V$ at $v$ is $V$ itself, and the derivative of a linear map $f$ at $v$ is the linear approximation of $f$, which is $f$ itself.  One gets $\adj \circ (\Ad g) = \Int_{\Ad g} \circ \adj$, which, evaluated at $\xi$ and then at $(\Ad g)(\zeta)$, gives $[(\Ad g)(\xi), (\Ad g)(\zeta)] = (\Ad g)[\xi, (\Ad g)^{-1}((\Ad g)(\zeta))]= (\Ad g)[\xi,\zeta]$.
\end{remark}

Recall that if $\mathfrak{g}$ is a Lie algebra with Lie bracket $[-]$, then an \emph{ideal} of $\mathfrak{g}$ is a subspace $\mathfrak{h}$ such that 
$[\mathfrak{g},\mathfrak{h}]\subset \mathfrak{h}$. We say $\mathfrak{g}$ is \emph{abelian} if $[\mathfrak{g},\mathfrak{g}]=\{0\}$. It is \emph{simple} if its only ideals are $\{0\}$ and $\mathfrak{g}$, and is \emph{semisimple} if its only abelian ideal is $\{0\}$. Furthermore, any finite-dimensional semisimple Lie algebra is a direct sum of finitely many ideals which are simple Lie algebras.

\begin{lemma}\label{2.31}
Let $\mathfrak{g}$ be the Lie algebra of the definable group $G$ and let $\mathfrak{h}$ be a subspace of $\mathfrak{g}$. Then
\begin{enumerate}
\item[(i)] The subalgebra $\{\xi\in \mathfrak{g}: [\xi,\mathfrak{h}]=0\}$ is the Lie algebra of the subgroup 
$\{g\in G: \Ad g|_{\mathfrak{h}}=\Id\}$, and

\item[(ii)] the subalgebra $\{\xi\in \mathfrak{g}: [\xi,\mathfrak{h}]\subseteq \mathfrak{h}\}$ is the Lie algebra of
the subgroup $\{g\in G: \Ad g(\mathfrak{h})\subseteq \mathfrak{h}\}$.
\end{enumerate}
\end{lemma}

\begin{proof} See \cite[Claim 2.31]{pps}. This follows by applying Corollary~\ref{2.24} to the adjoint representation of $G$.
\end{proof}

\begin{lemma} \label{2.26}
Let $G$ be a definable group, and $\Ad\colon G \to \Aut(T_e(G))$ be as above.
\begin{enumerate}
\item[(i)] 
For $x\in G$, we have
$x\in \Ker(\Ad)\Leftrightarrow |G:C_G(x)|<\infty$.
\item[(ii)] Assume $G$ is a definable subgroup of some $\GL_n(K)$. Then $\Ker(\Ad)$  centralises a definable normal subgroup of $G$ of finite index.
\end{enumerate}
\end{lemma}

\begin{proof} (i) is exactly (the first statement of) \cite[Lemma 8.7]{AH}.

(ii) It is well-known that any subgroup of a group $\GL_n(K)$  has the descending chain condition on centralisers; for example, this holds for the group $\GL_n(K^{\alg})$ because the latter is  stable  (see e.g. \cite [Corollary 1.0.7]{wagner}), and is inherited by subgroups of $\GL_n(K^{\alg})$ since the formula $xy=yx$ is quantifier-free.  

Let $g\in \Ker(\Ad)$. Then  $C_G(g)$ is open in $G$ of finite index by (i).  Put $N:=\bigcap_{g\in \Ker(\Ad)} C_G(g)$, a definable normal subgroup of $G$ with $\Ker(\Ad)\leq C_G(N)$. As $G$  has the descending chain condition on centralisers, we have $N=C_G(F)$ for some finite $F\subset \Ker(\Ad)$, so $|G:N|$ is finite as required.
\end{proof}

\begin{lemma}\label{2.32}
Let $G$ be a definable subgroup of $\GL_n(K)$ with Lie algebra $\mathfrak{g}$.
\begin{enumerate}
\item[(i)] $G$ has an abelian definable subgroup of finite index if and only if $\mathfrak{g}$ is abelian.  
\item[(ii)] Let $H$ be a definable subgroup of $G$ with Lie algebra $\mathfrak{h}$.  
Suppose that $H$ is commensurable with a definable normal subgroup $H_0$ of  a definable open subgroup $G_0$ of $G$. Then $\mathfrak{h}$ is an ideal of $\mathfrak{g}$.
\end{enumerate}
\end{lemma}

\begin{proof}
Suppose $\mathfrak{g}$ is abelian. By \cite[Lemma 8.7]{AH}, this implies that $|G:\Ker(\Ad)|$ is finite. 
By Lemma~\ref{2.26} (ii), $\Ker(\Ad)$ centralises a definable normal subgroup $H \subset G$ of finite index. The intersection $H \cap \Ker(\Ad)$ also has finite index in $G$, and it is abelian.

Conversely, suppose that $G$ has a definable abelian subgroup $U$ of finite index. Then for each $g\in U$ the group $C_G(g)$ is of finite index, so by Lemma~\ref{2.26}(i), for all $g\in U$ we have $\Ad g=1$, so $\Ker(\Ad)\geq U$. It follows by Lemmas~\ref{2.31}(i) and \ref{opensubgroup} that $L(\Ker(\Ad))=\{\xi\in \mathfrak{g}:[\xi,\mathfrak{g}]=0\}=L(G)$. Hence $\mathfrak{g}$ is abelian.

(ii) Let $G_0$, $H$, and $H_0$ be as in the statement, and let $\mathfrak{h}=L(H)$ and $\mathfrak{h}_0=L(H_0)$. First observe that $\mathfrak{h}=\mathfrak{h}_0$ by Lemma~\ref{2.20}.
Then for every $g\in G_0$, $\Int_g(H_0)=H_0$, so $\Ad g(\mathfrak{h})=\mathfrak{h}$ (e.g. by Corollary~\ref{2.22}). 
Let $G_1:=\{g\in G: \Ad g(\mathfrak{h})\subseteq\mathfrak{h}\}$. Then $G_0\leq G_1\leq G$,
and $L(G_1)=\{\xi\in \mathfrak{g}: [\xi,\mathfrak{h}]\subseteq\mathfrak{h}\}$ by Lemma~\ref{2.31}(ii). Also, by Lemma~\ref{opensubgroup}, $L(G_0)=L(G_1)=L(G)$ as $G_0$ is open in $G$, so $\mathfrak{h}$ is an ideal of $\mathfrak{g}$.
\end{proof}

\begin{theorem} \label{2.34}
Let $G$ be a definable subgroup of $\GL_n(K)$ with Lie algebra $\mathfrak{g}$. Then the following are equivalent.

(i) $G$ has an infinite definable abelian normal subgroup.

(ii) $G$ has a definable open subgroup $H$ with an infinite definable abelian normal subgroup.

(iii) $\mathfrak{g}$ is not semisimple.

\end{theorem}

\begin{proof} We adapt \cite[Theorem 2.34]{pps}. The direction (i) $\Rightarrow$ (ii) is trivial. 

(ii)  $\Rightarrow$ (iii).  Suppose that $G$ has an open subgroup $H$ which has an infinite definable abelian normal subgroup $B$. We apply Lemma~\ref{2.32}: part (ii) yields that  $L(B)$ is a non-trivial ideal of $\mathfrak{g}$, and it is abelian as $B$ is abelian.

(iii) $\Rightarrow$ (i) Suppose that $\mathfrak{g}$ is not semisimple. Then it has an abelian non-trivial ideal, so has non-trivial soluble radical, and the last non-trivial term in its derived series is an abelian ideal  $\mathfrak{j}$ which is invariant under $\Aut(\mathfrak{g})$. 

 Put $H:=\{g\in G: \Ad g|_{\mathfrak{j}}=\Id\}$. By Lemma~\ref{2.31}(i), 
$\mathfrak{h}:=L(H)=\{\xi\in \mathfrak{g}: [\xi,\mathfrak{j}]=0\}$. Then $\mathfrak{h}$ is an ideal of $\mathfrak{g}$ (use the Jacobi identity), and contains $\mathfrak{j}$ as the latter is abelian. Also, by the above invariance, for every $g\in G$ we have $\Ad g(\mathfrak{j})=\mathfrak{j}$, 
so $g^{-1}Hg=H$ from the definition of $H$, so $H$ is normal in $G$. 
Now let $B:=\{h\in H\colon  \Ad h|_{\mathfrak{h}}=\Id\}$. Then $B$ is normal in $G$ (as $\mathfrak{h}$ is  $\Aut(\mathfrak{g})$-invariant), and $L(B)=\{\xi\in \mathfrak{h}: [\xi,\mathfrak{h}]=0\}$ by Lemma~\ref{2.31}(i). Now by the Jacobi identity, $L(B)$ is an ideal of $\mathfrak{g}$, and contains $\mathfrak{j}$, so $B$ is infinite. Furthermore, the adjoint action of $B$ on $L(B)$ is trivial, so
(by Lemma~\ref{2.26} (i)) for every $g\in B$ we have $|B:C_B(g)|$ finite. It follows by  the descending chain condition on centralisers for $G$ that $|B:Z(B)|$ is finite. Thus, $Z(B)$ is an infinite definable abelian normal subgroup of $G$.
\end{proof}

\begin{theorem}\label{simpleliealg}
Suppose that $G$ is a definably almost simple non-abelian definable subgroup of $\GL_n(K)$. Then its Lie algebra $\mathfrak{g}$ is simple.
\end{theorem}

\begin{proof} By Theorem~\ref{2.34}, $\mathfrak{g}$ is semisimple, so is a direct sum of a unique  set $\{\mathfrak{h}_1,\ldots,\mathfrak{h}_t\}$ of ideals which are non-abelian and simple as Lie algebras. We suppose for a contradiction that $t>1$.

Let $\mathfrak{j}:=\mathfrak{h}_2 \oplus \ldots \oplus \mathfrak{h}_t$. Then $\mathfrak{h}_1=\{\xi\in \mathfrak{g}:[\xi,\mathfrak{j}]=0\}$, and by Lemma~\ref{2.31}(i), $\mathfrak{h}_1$ is the Lie algebra of $N_1:=\{g\in G: \Ad g|_{\mathfrak{j}}=\Id\}$. Note that $G$ has no proper definable  subgroup $N$ of finite index, since otherwise $\bigcap_{g\in G}g^{-1}Ng$ is a proper definable {\em normal} subgroup of $G$ of finite index, contradicting definable almost simplicity. Since  the collection $\{\mathfrak{h}_1,\ldots,\mathfrak{h}_t\}$ is uniquely determined, it follows that each $\mathfrak{h}_i$ is $G$-invariant under the adjoint action. Thus, $N_1$ is a normal subgroup of $G$, and is clearly definable. Also $N_1$ is infinite
 as $\mathfrak{h}_1$ is non-trivial, so $N_1=G$ by definable almost simplicity. Since we may replace $\mathfrak{h}_1$ by any $\mathfrak{h}_j$, it follows that the adjoint representation of $G$ is trivial, so $|G:Z(G)|$ is finite by Lemma~\ref{2.26}(ii). 
This contradicts the assumption that $G$ is non-abelian and definably almost simple.
\end{proof}

\begin{lemma}\label{al-simp}
Let $G$ be an infinite non-abelian definable and definably almost simple subgroup of $\GL_n(K)$. Then $G\leq \SL_n(K)$.
\end{lemma}
\begin{proof} Suppose that $G$ is not contained in $\SL_n(K)$. Then there is a definable non-trivial homomorphism $\sigma:G \to \GL_n(K)/\SL_n(K) \cong (K^\times,\cdot)$. As $\Ker(\sigma)$ is a definable proper normal subgroup of $G$, it is  finite, and its centraliser equals $G$, so $\Ker(\sigma)$ is central (see Remark~\ref{rem:simple}), and $G$ is nilpotent of class 2. Now for any $g\in G$, $C_G(g)$ has finite index in $G$ so equals $G$. It follows that $G$ is abelian, a contradiction.
\end{proof}

The following result pulls together what is needed from this section. 

\begin{proposition} \label{linearityproof}
Let $(K,v,\ldots)$ be a 1-$h$-minimal valued field, and let $G \subseteq K^n$ be a definable definably almost simple non-abelian group, equipped with the dwsd Lie group structure as in Proposition~\ref{p.Gman}. Then there is $d\in \N$ and a definable homomorphism $\rho\colon  G \to \SL_d(K)$ with  kernel $Z(G)$, given by the adjoint action of $G$ on the tangent space $T_e(G)$. Moreover, $T_e(G)$ carries the structure of a simple Lie algebra $\mathfrak{g}$ over $K$, and $\rho(G)$
 acts as a  group of automorphisms of $\mathfrak{g}$.  We have $\dim(G)=\dim_K(\mathfrak{g})$.
\end{proposition}

\begin{proof}
For $g\in G$, let ${\rm Int}_g\colon G \to G$ be the map $x \mapsto gxg^{-1}$.
 Then the differential $d_e {\rm Int}_g$ is a definable linear map $T_e(G) \to T_e(G)$,
and gives the adjoint representation ${\rm Ad}\colon  G \to \GL_d(K)$, defined by putting $\Ad(g)=d_e {\rm Int}_g$.

Suppose first for a contradiction that $G=\Ker(\Ad)$. Then for each $x\in G$, $C_G(x)$ has finite index in $G$ (by Lemma~\ref{2.26} (i)); hence $C_G(x)=G$ by definable almost simplicity of $G$, for otherwise $\bigcap_{g\in G}g^{-1}C_G(x)g$ is a proper definable normal subgroup of $G$. Thus $G$ is abelian, a contradiction.

Thus,  ${\rm Ad}$ has finite  kernel containing $Z(G)$, and $C_G(\Ker(\Ad))$ has finite index in $G$ so equals $G$ by definable almost simplicity, so $\Ker(\Ad)=Z(G)$.
Since by Lemma~\ref{al-simp} ${\rm Ad}$ has image in $\SL_n(K)$), we obtain the first assertion, putting $\rho:=\Ad$. 
Furthermore, by Lemma~\ref{2.27}, the tangent space  of $G$ at the identity carries the structure of a  Lie algebra $\mathfrak{g}=L(G)$ of $G$, and by Theorem~\ref{simpleliealg}, $\mathfrak{g}$ is a simple Lie algebra. In particular, we have $\rho(G)\leq\Aut(\mathfrak{g})$.

Finally, $\dim(G)=\dim(T_e(G))$, and this equals the $K$-vector space dimension of $\mathfrak{g}$; see e.g. Remark~\ref{rem:manichoices}(3).
\end{proof}

\section{Bounded subgroups}
\label{sec:bounded}

The purpose of this section is to present and adapt results from \cite{prasad} which are key to our proof, and (in Proposition~\ref{prop:main}) to eliminate the possibility that the group $G$ in Theorem~\ref{main} is `bounded'.

If $(K,v)$ is any valued field, we shall say that a subset of $K^n$ is \emph{bounded} if it is contained in the ball of $K^n$ of form $B_{\gamma}(0)$ around 0, for some $\gamma\in\Gamma$ (see also \cite[Section 3]{bt}). For example 
the local ring $\O$ of the valuation is a bounded set.
We say that a subgroup $G \subset \mathbf H(K)$ is \emph{bounded}, where $\mathbf{H}$ is a linear algebraic $K$-group, if $G$ is bounded as a subset of $K^{n^2}$ when we consider $\mathbf{H}$ as being embedded into $\SL_n$. Note that we have a running assumption that the valuation $v$ is surjective, so if a set $B$ is bounded then  there is $a\in K$ such that for all $b\in B$ we have $v(a)<v(b)$. 

\begin{remark} \label{boundedness} Note that this notion of boundedness of $G$ is well-defined in the sense that it does not depend on the embedding of  $\mathbf{H}$ into $\SL_n$.
Indeed, if $\mathbf H \subset \SL_n$ and $\mathbf H' \subset \SL_{n'}$ are linear algebraic $K$-groups and $f\colon \mathbf H \to \mathbf H'$ is a $K$-isomorphism, then a bound on $G \subset \mathbf H(K)$ yields a bound on $f(G) \subset \mathbf H'(K)$ (since $f$ is defined by polynomials with coefficients in $K$), and similarly for $f^{-1}$.
\end{remark}

\subsection{Bounded subgroups of $\mathbf{H}(K)$}

In this subsection, we record the following two theorems, which  are crucial to our proof of Theorem~\ref{main}. In the first theorem, we have added the assumption that $K$ is perfect for consistency with our definitions in this paper.

\begin{theorem}[{\cite[Theorem (BTR) on p.~198]{prasad}}] \label{thm:prasadi}
Suppose that $K$ is a perfect field with a non-trivial
henselian valuation $v\colon K \to \Gamma\cup\{\infty\}$
(with possibly non-archimedean value group $\Gamma$)
and $\mathbf{H} \le \SL_n$ is a connected reductive linear algebraic $K$-group.
Then $\mathbf{H}(K)$ is bounded if and only if $\mathbf{H}$ is anisotropic over $K$.
\end{theorem}

\begin{theorem}\label{thm:prasadii}
Suppose $(K,v,\ldots)$ is a 1-$h$-minimal expansion of a henselian valued field of characteristic 0. Let
$\mathbf{H} \le \SL_n$ be a semisimple almost $K$-simple algebraic $K$-group.
Then any unbounded definable open subgroup $G$ of $\mathbf{H}(K)$ 
contains $\mathbf{H}(K)^+$.
\end{theorem}

The context in \cite{prasad} is that of Henselian valuations whose value group $\Gamma$ is archimedean, that is, has the property that if $\gamma\in \Gamma$ with $\gamma>0$ then $\{n\gamma:n\in \mathbb{N}\}$ is cofinal in $\Gamma$. We need these statements for arbitrary value group, \emph{i.e.} for a {\em Krull valuation}. 
The proof of the above Theorem~\ref{thm:prasadi} in \cite{prasad} (attributed there to Bruhat, Tits, and Rousseau) does not use the archimedean value group  assumption (see also the proof of Theorem 1.1 of \cite{prasad2}). For Theorem~\ref{thm:prasadii}, which is based on a proof in \cite{prasad} of an unpublished result of Tits, the situation is slightly more complicated. Our conclusion (that $G$ contains $\mathbf{H}(K)^+$) is slightly stronger than that in \cite{prasad}, and moreover, Prasad's proof really seems to be using the assumption that the value group is archimedean.  It is not very difficult to modify Prasad's proof to get the stronger conclusion, and it is also not difficult to generalize it to value groups which have a maximal proper convex subgroup (since then, the valuation can be coarsened to an archimedean one). However, getting the result for general value groups seems to be more tricky. We at least were only able to prove the general case under the assumption that $G$ is definable -- an assumption which Prasad does not need.

We also need the following positive characteristic version of Theorem~\ref{thm:prasadii}:

\begin{proposition}\label{prop:prasadiii}
Suppose that $(K,v)$  is a pure algebraically closed valued field of characteristic $p$. 
Let
 $\mathbf{H} \le \SL_n$ be a semisimple almost $K$-simple algebraic $K$-group.
 Then any proper definable open subgroup $G$ of $\mathbf{H}(K)^+$ is bounded.
\end{proposition}

Here, the situation is easier than for Theorem~\ref{thm:prasadii}: firstly, we only need the weaker conclusion which also appears in \cite{prasad}, and secondly, using that we only work in the pure valued field language, we can easily reduce to the case of the value group having a maximal proper convex subgroup, so that we can then cite \cite{prasad}.

\begin{proof}[Proof of Proposition~\ref{prop:prasadiii}]
Suppose for a contradiction that $G$ is an unbounded proper definable open subgroup of 
$\mathbf{H}(K)^+$. Let $\Gamma_1$ be a divisible ordered abelian group containing $\Gamma$ and having a maximal proper convex subgroup $\Gamma_2$, and let $\Delta=\Gamma_1/\Gamma_2$, an archimedean ordered abelian group. Also let $k$ be the residue field of $K$. By model-completeness in $L_{{\rm div}}$, the field $(K,v)$ is an elementary substructure of the Hahn field $(k((t^{\Gamma_1})),v_1)$ where $v_1$ is the natural valuation with value group $\Gamma_1$. Using that $G$ is definable, that $\mathbf{H}(K)^+$ is definable when $K$ is algebraically closed (see Lemma~\ref{lem:kneser}(ii)) and that the conditions `bounded' and `open' are first order expressible, we may replace $K$ by $k((t^{\Gamma_1}))$, that is, we may assume $\Gamma=\Gamma_1$. 

Now let $v_\Delta$ be the induced coarsened valuation $v_\Delta:K\to \Gamma_1\to \Delta$. The group $G$ is still an unbounded and open subgroup of $\mathbf{H}(K)^+$ with respect to this valuation, and is also a proper subgroup. This contradicts Theorem (T) (the theorem attributed by Prasad to Tits) on \cite[p.~198]{prasad}. 
\end{proof}

The remainder of this subsection is devoted to the proof of Theorem~\ref{thm:prasadii}. We start with a preliminary lemma which is probably known, but since we didn't find it in the literature, we give a proof.

\begin{lemma}\label{iso-ext}
Suppose that $K_1$ and $K_2$ are two spherically complete valued fields of equi-characteristic $0$, which both have the same value group $\Gamma$ and the same residue field $k$. Suppose moreover that $K_1$ and $K_2$ have a common subfield $K_0$. Then there exists a valued field isomorphism $f\colon K_1 \to K_2$ which is the identity on $K_0$ and which induces the identity on $\Gamma$ and on $k$.
\end{lemma}

\begin{proof}
Consider all the fields as two-sorted structures, with one sort for the valued field and one sort for the leading terms. We set $\RV := \rv(K_1) = \rv(K_2)$ and $\RV_0 := \rv(K_0)$. 
Let $f$ be the identity map on $K_0 \cup \RV$ (considered as a partial isomorphism from $(K_1, \RV)$ to $(K_2, \RV)$). It suffices to verify
that given any $a \in K_1$, we can extend $f$ to an embedding
$(K_0(a), \RV) \to (K_2, \RV)$. Indeed, then this process can be repeated to obtain an embedding $(K_1, \RV)$ to $(K_2, \RV)$, and if the image $f(K_1)$ would be a proper subfield of $K_2$, then applying the same argument to $f^{-1}$ and some $b \in K_2 \setminus f(K_1)$ would lead to a contradiction.

So let $a \in K_1 \setminus K_0$ be given, and suppose first that $a \in \acl(K_0, \RV)$. Then, since $K_1 \equiv_{K_0\cup\RV} K_2$ (see for example the relative quantifier elimination to the valued field sort in \cite[Proposition 4.3]{flenner}), there exists $a' \in K_2$ with $\tp(a'/K_0\cup\RV) = \tp(a/K_0\cup\RV)$, so we can extend $f$ by setting $f(a) := a'$. Thus we may now assume that $K_0 \cup \RV$ is algebraically closed in the model theoretic sense.
In that case, the type $\tp(a/K_0\cup\RV)$ is determined by the formulas of the form $\rv(x - b) = \xi$, where $b$ runs over $K_0$ and $\xi$ runs over $\RV$. (This follows e.g. from \cite[Lemma~2.4.4]{iCR.hmin}, which states that
for $a, a' \in K_1$ to have the same type over $K_0 \cup \RV$, it suffices that the smallest ball $B \subset K_1$ containing $a$ and $a'$ is disjoint from $\acl(K_0) = K_0$. If there would exist a $b \in B \cap K_0$, then we would have $\rv(a - b) \ne \rv(a'-b)$.) Each such formula defines a valuative ball in $K_1$, and those balls form a chain. The same formulas then also define a chain of balls in $K_2$. By spherical completeness of $K_2$, that chain has non-empty intersection. Any $a'$ in that intersection has the same type as $a$ over $K_0 \cup \RV$, so that we can define $f(a) := a'$.
\end{proof}

Here is another preliminary lemma:

\begin{lemma}\label{lou2}
Let $K$ be a henselian valued field of characteristic $0$ (of any residue characteristic).
Let $\mathbf H$ be a connected linear algebraic $K$-group and 
let $U$ be a valuation-open subset of $\mathbf{H}(K)$. Then $U$ is Zariski-dense in $\mathbf{H}(K)$.
\end{lemma}

\begin{proof}
Without loss, $\mathbf H = \cl_{K^{\alg}}(\mathbf H(K))$.
By Lemma~\ref{l.locdim}, there exists an $x \in \mathbf{H}(K)$ such that for every ball $B \subset \mathbf{H}(K)$ containing $x$, we have $\dim B = \dim \mathbf{H}(K)$. By left-translation, we may assume $x \in U$, so using Lemma~\ref{lou}(ii) we obtain $\dim U = \dim \mathbf{H}(K) = \zardim \mathbf H$ and in particular also $\zardim \cl_{K^{\alg}}(U) = \zardim \mathbf H$. Since $\mathbf{H}$ is connected and hence irreducible, 
this implies $\cl_{K^{\alg}}(U) = \mathbf H$ and hence, using Remark~\ref{note-on-dim}, that $\cl_K(U) = \mathbf H(K)$.
\end{proof}

\begin{remark}\label{rem:unipot-gen}
Note that in Theorem~\ref{thm:prasadii}, $\mathbf{H}(K)^+$ is generated by the set $U$ of unipotent elements of $\mathbf{H}(K)$ (see Remark~\ref{def:defg+}). In particular, a subgroup $G \le \mathbf{H}(K)$ contains $\mathbf{H}(K)^+$ if and only if it contains $U$. Note also that $U$ can be defined by polynomials with coefficients in $K$, since it is the set of those matrices $A\in \mathbf{H}(K)$ such that $A-1$ is nilpotent, and since we can bound the nilpotency class of matrices in $\mathbf H$.
\end{remark}

\begin{proof}[Proof of Theorem~\ref{thm:prasadii}]

Let $G$ be a definable unbounded open subgroup of $\mathbf{H}(K)$, where $\mathbf{H}$ is a semisimple almost $K$-simple linear algebraic $K$-group with $\mathbf{H}(K^{\alg})\leq\SL_n\left(K^{\alg}\right)$. Our goal is to prove that $G$ contains $\mathbf{H}(K)^+$. We start by a series of reductions.

Firstly, note that we may suppose that $\mathbf{H}$ has trivial centre. Indeed, the central isogeny $\pi: \mathbf{H} \to \bar{\mathbf{H}}:=\mathbf{H}/Z(\mathbf{H})$ takes $\mathbf{H}(K)^+$ to $\bar{\mathbf{H}}(K)^+$ (see Remark~\ref{rem:kneser} (i)), takes open subgroups to open subgroups, and satisfies that the preimage of a bounded subgroup is bounded.

Before the next reduction, observe that we may freely replace $(K,v,\ldots)$ by an elementary extension. For by definability of $G$, if there is a counterexample for $(K,v\ldots)$ then there is a counterexample in the elementary extension. Indeed, we can express the condition $\mathbf{H}(K)^+\leq G$ using the parameters defining $\mathbf{H}$ and $G$ using Remark~\ref{rem:unipot-gen}.

Next, we give an argument to reduce the mixed characteristic case to that of residue characteristic 0. So suppose $(K,v,\ldots)$ is 1-$h$-minimal of mixed characteristic. By the previous paragraph, we may suppose that it is $\omega_1$-saturated. Let $\mathcal{O}_1$ be the finest valuation subring of $K$ containing $\mathcal{O}$ and $\mathbb{Q}$. There is a corresponding non-trivial convex subgroup $\Delta$ of $\Gamma$ such that $\mathcal{O}_1$ is the valuation ring with value group
$\Gamma_1=\Gamma/\Delta$, and corresponding valuation $v_1:K \to \Gamma_1$. We expand the language by a predicate for $\mathcal{O}_1$ and we now consider $K$ as a valued field with the valuation $v_1$ instead of $v$. Note that by construction of $\mathcal{O}_1$, $(K, v_1)$ is of equicharacteristic $0$, and by \cite[Theorem~2.2.8]{hmin2} $(K, v_1,\dots)$ with the expanded language is also $1$-h-minimal. Since $G$ remains unbounded if we replace $v$ by $v_1$, the result in the original structure $(K, v, \dots)$ follows from the result in the expanded structure $(K, v_1, \dots)$.
Thus, we may assume that
$(K,v,\ldots)$ is equicharacteristic 0.

Before we proceed to the next reduction, note that $K$ has an elementary extension which is a Hahn field, \emph{i.e.}, $K = k((t^\Gamma))$. Indeed,
this is essentially the statement of \cite[Theorem 1]{bw-h}. (The theorem states that a spherically complete elementary extension exists; in equicharacteristic $0$, every spherically complete valued field is already a Hahn field.)

Clearly, we may now assume that $K = k((t^\Gamma))$, but we want even more, namely that $\mathbf H$ is defined over a subfield $K' = k((t^{\Gamma'}))$, where $\Gamma'$ is a proper convex subgroup of $\Gamma$. To this end, let us first rename our current $K = k((t^\Gamma))$ to $K' = k((t^{\Gamma'}))$.
We then choose an elementary extension $K \succ K'$ with value group $\Gamma$, in such a way that $\Gamma'$ is contained in a proper convex subgroup of $\Gamma$. (To achieve this, just realize the type at $\infty$ over $\Gamma'$.) We may additionally assume by the last paragraph that $K$ is isomorphic to $k((t^\Gamma))$. (While this means that we can assume $K = k((t^\Gamma))$, note that if we do so, we only know that $\mathbf H$ is defined over a subfield $K'$ isomorphic to $k((t^{\Gamma'}))$, and not yet that we can take $K'$ to be equal to the subfield $k((t^{\Gamma'}))$ of $k((t^{\Gamma}))$.)
We apply Lemma~\ref{iso-ext} to $K_1 = K$ and $K_2 = k((t^\Gamma))$, where we identify the subfield $K' \subset K$ with
the subfield $k((t^{\Gamma'})) \subset k((t^\Gamma))$. We then use the obtained valued field isomorphism $f\colon K \to k((t^\Gamma))$ to transfer the additional structure from $K$ to $k((t^\Gamma))$,
so that we have $k((t^{\Gamma'})) \prec k((t^\Gamma))$ for the natural embedding.

Let us now fix some notation for the remainder of the proof:
Given $\lambda \in \Gamma$ with $\lambda>0$, we define:
\begin{itemize}
    \item $C_\lambda \subset \Gamma$ is the smallest convex subgroup of $\Gamma$ containing $\lambda$; $C_\lambda^- \subset C_\lambda$ is the largest convex subgroup of $\Gamma$ not containing $\lambda$.
    \item $\O_\lambda = \{a \in K \mid v(a) \in C_\lambda \vee v(a) > 0\}$ and $\O_\lambda^- = \{a \in K \mid v(a) \in C^-_\lambda \vee v(a) > 0\}$ are the valuation rings of $K$ corresponding to the value groups $\Gamma/C_\lambda$ and $\Gamma/C^-_\lambda$, respectively.
    \item $v_\lambda\colon K \to \Gamma/C^-_{\lambda} \cup \{\infty\}$ is the valuation with valuation ring $\O^-_\lambda$.
    \item $K_\lambda := k((t^{C_\lambda}))$ is the Hahn sub-field of $K$ corresponding to $C_\lambda$. Note that we have $K_\lambda \subset \O_\lambda$ and  that $K_\lambda$ is naturally isomorphic to the residue field of the valuation ring $\O_\lambda$. (The residue map restricts to an isomorphism from $K_\lambda$ to that residue field.)
    \item Set $G_\lambda := G \cap \mathbf H(K_\lambda)$.
\end{itemize}
Moreover, we let $K_\infty \subset K$ be the union of all the $K_\lambda$ (for $\lambda \in \Gamma$ positive). Note that $\mathbf H$ is defined over $K_\infty$, since we assumed it to be defined over some subfield $K' = k((t^{\Gamma'}))$, which is contained in $K_\lambda$ for any positive $\lambda \in \Gamma \setminus \Gamma'$.

While $K_\infty$ can be strictly smaller than $K$, it is always valuation dense in $K$. More precisely, we even have:

{\em Claim 1.} If $\mathbf X$ is any subvariety of $\mathbf H$ defined over $K_\infty$, then $\mathbf X(K_\infty)$ is valuation dense in $\mathbf X(K)$.

{\em Proof of Claim 1.}
We consider $\mathbf X$ as a $K_\infty$-closed subset of $\SL_n$.
As such, $\mathbf X$ is defined by some polynomials $f_i$ over $K_\infty$.  Let $a \in \mathbf X(K)$ and $\lambda \in \Gamma$ be given. We need to show that there exists an $a' \in \mathbf X(K_\infty)$ with $v(a-a') > \lambda$. By making $\lambda$ bigger, we can assume that
all coordinates of $a$ have valuation in $C_\lambda \cup \Gamma_{\ge 0}$ and that the polynomials $f_i$ have coefficients in
$K_\lambda$.

Write $a$ as a power series $\sum_{\gamma\in \Gamma} a_\gamma t^\gamma$, with $a_\gamma \in k^{m}$ and set $a' := \sum_{\gamma\in C_\lambda} a_\gamma t^\gamma$. Clearly, we have $a' \in K_\lambda^m$ and $v(a' - a) > \lambda$, so it remains to verify that $a'$ is a common zero of the polynomials $f_i$.
But this is clear if we identify $K_\lambda$ with the residue field of $\O_\lambda$: If we do so, and if we write $\res_\lambda\colon \O_\lambda \to K_\lambda$ for the corresponding residue map, then $a' = \res_\lambda(a)$, $f_i = \res(f_i)$, and hence $f_i(a') = \res(f_i)(\res(a)) = \res(f_i(a)) = 0$.\qed \emph{(Claim~1)}

We record the following intermediate result of the above proof for further usage:

{\em Claim 1'.} If $\mathbf X \subset \mathbf H$ is defined over $K_\lambda$, then for every $a \in \mathbf X(K)$ with $v(a) \in C_\lambda \cup \Gamma_{\ge 0}$, there exists an $a' \in \mathbf X(K_\lambda)$ with $v(a' - a) > \lambda$.

The main part of the proof of the Theorem consists in showing the following:
\begin{itemize}
    \item[(*)] There are arbitrarily big $\lambda$ such that
$G_\lambda$ contains $\mathbf H(K_\lambda)^+$.
\end{itemize}
(Note that $\mathbf H(K_\lambda)^+$ makes sense, since $\mathbf H$ is defined over $K_\infty$ and hence also over $K_\lambda$ for $\lambda$ big enough.)

Let us already verify that (*) implies the theorem:
By Remark~\ref{rem:unipot-gen}, it suffices to show that $G$ contains $\mathbf U(K)$, where $\mathbf U$ is the set of unipotent matrices in $\mathbf H$. By (*), $G$ contains $\mathbf U(K_\lambda)$ for all sufficiently big $\lambda$ and hence it contains $\mathbf U(K_\infty)$.
Since $G$ is open (and hence closed) in $\mathbf H(K)$ and $\mathbf U(K_\infty)$ is valuation dense in $\mathbf U(K)$ (by Claim~1), we indeed obtain $\mathbf U(K) \subset G$.

It remains to prove (*).
To this end, fix $\lambda \in \Gamma$ positive.
We write $B_\lambda(1)\cap \mathbf{H}(K)$ for the open ball of radius $\lambda$ around the identity of $\mathbf{H}(K)$.
We will always assume that $\lambda$ is so big that $B_\lambda(1)\cap \mathbf{H}(K)$ is contained in $G$ and also that $\mathbf{H}$ is defined over $K_\lambda$.

Our proof proceeds in a series of claims. We first (Claim 2) show that $G_\lambda$ is Zariski dense in $\mathbf{H}(K_\lambda)$. Then (Claim 3, mimicking \cite[Lemma 1]{prasad}) we show that $G_\lambda$ has semisimple elements which in the adjoint representation have eigenvalues of large negative valuation.
A key difference between (*) and the statement of Theorem~\ref{thm:prasadii} is that the valuation on $K_\lambda$ can be coarsened to a rank $1$ valuation. Working in $K_\lambda$ allows us to follow the strategy of \cite[Lemma 2]{prasad} to deduce (Claim 5) that $G_\lambda$ contains certain unipotent subgroups $\mathbf{U}_\Pi(K_\lambda)$ and $\mathbf{U}^-_\Pi(K_\lambda)$ (see definitions below). Using known results, this implies that $G_\lambda$ contains $\mathbf H(K_\lambda)^+$.

{\em Claim 2.}
The set $B := B_\lambda(1) \cap \mathbf H(K_\lambda)$ is Zariski dense in $\mathbf{H}(K_\lambda)$. In particular, for sufficiently big $\lambda$, $G_\lambda$ is Zariski dense in $\mathbf{H}(K_\lambda)$, since $G_\lambda \supset B$.

{\em Proof of Claim 2.}
The first claim follows directly from Lemma~\ref{lou2}, applied in the field $K_\lambda$.
That $G_\lambda$ contains $B$ for big $\lambda$ follows from $G$ being open in $\mathbf{H}(K)$.
\qed \emph{(Claim~2)}

Let $\Ad$ denote the adjoint representation of $\mathbf{H}(\Klamalg)$, viewed as a $K_\lambda$-morphism of $K_\lambda$-groups $\mathbf{H}(\Klamalg) \to \SL(V)$, where $V$ is a $\Klamalg$ vector space.
Since we assumed $\mathbf{H}$ to be centreless (at the very beginning of the proof), $\Ad$ is injective. To see that the image is in $\SL$, see the proof of Lemma~\ref{al-simp}. 

In the following, we use the natural extension of the valuations $v$ and $v_\lambda$ on $K$ to the algebraic closure $K^{\alg}$.

{\em Claim 3.} (i)
There exist arbitrarily large $\lambda \in \Gamma$ such that there is $g \in G_\lambda$ such that $\Ad(g)$ has an eigenvalue $\alpha \in K_\lambda^{\alg}$ with $v_\lambda(\alpha) < 0$.

(ii) The element $g$ in (i) can be chosen to be semisimple.

{\em Proof of Claim 3.}
(i) Fix a flag
\[V=V_0\supset V_1 \supset \dots \supset V_{r+1}=\{0\},\]
of sub-representations $V_i$ such that for each $i=0,\ldots,r$,  the natural representation $\rho_i$ of $\mathbf H(\Klamalg)$ on $W_i:=V_i/V_{i+1}$ is irreducible.
Let $\rho=\bigoplus_{i=0}^r \rho_i$ be the corresponding representation of $\mathbf H(\Klamalg)$ on $\bigoplus_{i=0}^r W_i$. Since $\mathbf H(\Klamalg)$ is reductive and the kernel of $\rho$ is a unipotent normal subgroup of $\mathbf H(\Klamalg)$, $\rho$ is injective: indeed, with respect to a basis of $V$ obtained by iteratively extending a basis of each $V_{i+1}$ to $V_i$, each matrix in the kernel is upper unitriangular. By Remark~\ref{rem:bijmorph}, injectivity of $\rho$ implies that it induces an isomorphism of $\mathbf H(\Klamalg)$ onto its image $\rho(\mathbf H(\Klamalg))$.

Let us temporarily fix $i \le r$. 
Since $\rho_i$ is irreducible,
$\rho_i(\mathbf H(\Klamalg))$ contains a basis $B_i = \{b_{i,1},\dots, b_{i,m_i}\}$ of the $\Klamalg$-algebra ${\rm End}_{\Klamalg}(W_i)$ of $\Klamalg$-vector space endomorphisms of $W_i$ (by Jacobson's Density Theorem; see e.g. Corollary 3.3 in Ch. XVII of \cite{lang}). Since
$\rho_i(G_\lambda)$ is Zariski dense in $\rho_i(\mathbf H(\Klamalg))$ (by Claim~2), this is also true for the $m_i$-th cartesian powers:
$\rho_i(G_\lambda)^{m_i}$ is Zariski dense in $\rho_i(\mathbf H(\Klamalg))^{m_i}$. Since being a basis is a Zariski open condition, we may pick our above basis $B_i$ to be a subset of $\rho_i(G_\lambda)$.
Recall that on ${\rm End}_{\Klamalg}(W_i)$, we have a symmetric non-degenerate bilinear form $(a, a') \mapsto \tr(aa')$ called the trace form; let $B^*_i = \{b^*_{i,1},\dots, b^*_{i,m_i}\}$ be the basis of ${\rm End}_{\Klamalg}(W_i)$ dual to $B_i$ with respect to this trace form.

Now consider any $\lambda' \ge \lambda$ and set $W'_i := W_i \otimes_{\Klamalg} \Klampalg$. Then $\rho_i(G_{\lambda'})$
still contains $B_i$, which we now consider as a $\Klampalg$-basis of 
${\rm End}_{\Klampalg}(W'_i)$, and $B^*_i$ is still its dual basis with respect to the trace form. In particular, for any endomorphism $h \in {\rm End}_{\Klampalg}(W'_i)$, we have $h = \sum_j \tr(hb_{i,j})b_{i,j}^*$. Applying this to each $h \in \rho_i(G_{\lambda'})$ yields 
\[
\rho_i(G_{\lambda'}) \subset \sum_{j=1}^{m_i} \tr(\rho_i(G_{\lambda'}))b^*_{i,j}.\tag{**}
\]
Suppose now that Claim~3 (i) does not hold for $\lambda'$, \emph{i.e.}, that all the eigenvalues $\alpha$ of $\Ad(g)$ for all $g \in G_{\lambda'}$ satisfy $v_{\lambda'}(\alpha) \ge 0$. Since $\Ad(g)$ has the same eigenvalues as $\rho(g)$, we obtain $v_{\lambda'}(\tr(\rho_i(G_{\lambda'}))) \ge 0$.

To get to a contradiction, first consider the case that $C_\lambda = \Gamma$ (and hence $K_\lambda =K_{\lambda'} = K_\infty = K$ and $\O_\lambda^- = \O_{\lambda'}^-$). Since $G$ is an unbounded subgroup of $\mathbf{H}(K^{\alg})$,
$\rho(G)$ is unbounded, too, so there is some $i \le r$ such that $\rho_i(G)$ is unbounded.
In particular, the left hand side of (**) is unbounded, whereas the right hand side is bounded, which is a contradiction. We may therefore now assume that
$C_\lambda$ is strictly contained in $\Gamma$. In particular, we may assume that $\lambda'$ lies outside of $C_\lambda$ (and prove Claim~3 (i) only for such $\lambda'$).
This in particular implies that $\lambda \in C_{\lambda'}^-$ and hence $K_\lambda \subset \O_{\lambda'}^-$.
Moreover, we have
$v_{\lambda'}(b^*_{i,j}) \ge 0$ for all $i$ and $j$, so $v_{\lambda'}$ of the entire sum on the right hand side of (**) is non-negative. To finish the proof of Claim~3, it suffices to show that for arbitrarily large $\lambda'$, we can find a $g \in G_{\lambda'}$ and an $i \le r$ such that $v_{\lambda'}(\rho_i(g)) < 0$.

Since $G$ is unbounded in $\mathbf H(K)$, there exist
$g \in G$ with arbitrarily negative valuation. Pick such $g$ and set $\lambda' := -v(g)$. By Claim~1', there exists a $g' \in \mathbf H(K_{\lambda'})$ such that $v(g - g') > \lambda'$, and since 
$B_\lambda(1) \cap \mathbf H(K) \subset G$, we have $g' \in G$. Thus we may as well assume $g \in G_{\lambda'}$.
By definition of $v_{\lambda'}$, we have $v_{\lambda'}(g) < 0$.
Since $\rho$ is a $K_\lambda^{\alg}$-isomorphism onto its image (by Remark~\ref{rem:bijmorph}) and
all elements $a \in K_\lambda^{\alg}$ satisfy $v_{\lambda'}(a) \ge 0$, we also have $v_{\lambda'}(\rho(g)) < 0$; otherwise, applying $\rho^{-1}$ to $\rho(g)$ would imply $v_{\lambda'}(g) = 0$.
Using that $\rho$ is the direct sum of the $\rho_i$, we deduce that there exists an $i$ for which we have $v_{\lambda'}(\rho_i(g)) < 0$, as desired. This finishes the proof of Claim~3 (i).

To see (ii), recall that a semisimple element $h$ of $\mathbf{H}(\Klamalg)$ is said to be {\em regular} if $h$  lies in a unique maximal torus of $\mathbf{H}(\Klamalg)$.
By Theorem 2.14 of \cite{steinberg}, the set $\mathbf{H^{{\rm rss}}}(\Klamalg)$ of regular semisimple elements of $\mathbf{H}(\Klamalg)$ is dense open in $\mathbf{H}(\Klamalg)$. In particular, the set $\mathbf{H^{{\rm nss}}}(K_\lambda)$ of non-semisimple elements of $\mathbf{H}(K_\lambda)$ is contained in a proper Zariski closed subset of $\mathbf{H}(K_\lambda)$.
It follows that for $B := B_\lambda(1) \cap \mathbf H(K_\lambda) \subset G_\lambda$ as in Claim~2 (and $g$ as obtained from (i)), 
the subset $gB$ of $G_\lambda$ contains a semi-simple element $g'$; indeed, $B$ is Zariski-dense in $\mathbf{H}(K_\lambda)$ by Lemma~\ref{lou2}, and hence so is $gB$.
Using the continuity of $\Ad$, we  may choose $g'$  close to $g$, so that $\Ad(g')$ is sufficiently close to $\Ad(g)$ to ensure 
that the eigenvalues of $\Ad(g)$ and $\Ad(g')$ have the same valuations. (Here we use the well-known fact that the roots of the characteristic polynomial depend continuously on the characteristic polynomial -- see e.g. \cite[Theorem 23]{fvk}.) 
In other words, we may suppose that $g$ is semisimple, completing the proof of the claim.
\qed
\emph{(Claim~3)}

\medskip

For the remainder of the proof, fix $\lambda \in \Gamma$ and $g \in G_\lambda$ as in Claim 3. From now on, we only work in $K_\lambda$ and $K^{\alg}_\lambda$, 
and we only use the coarsened and restricted valuation $v_\lambda \colon K_\lambda \to C_\lambda/C^-_\lambda \cup \{\infty\}$ and its (unique) extension to $K^{\alg}_\lambda$, which we also denote by $v_\lambda$.

Following \cite[Lemma 2]{prasad}, we fix a maximal torus $\mathbf{T}$ of $\mathbf{H}$ defined over $K_\lambda$ and with $g\in \mathbf{T}(K_\lambda)$; such $\mathbf{T}$ exists by Theorem 13.3.6 and Corollary 13.3.8 of \cite{springer}.
We denote by $X(\mathbf T)$ the character group of $\mathbf{T}$ and by $\Phi \subset X(\mathbf T)$ the root system of $\mathbf H$ with respect to $\mathbf{T}$.
One easily verifies that the map $X(\mathbf T) \to C_\lambda/C_\lambda^-$ defined by $\beta \mapsto v_\lambda(\beta(g))$ is a group homomorphism.
In particular, the set $\{\beta \in X(\mathbf T) \mid v_\lambda(\beta(g)) \ge 0\}$ contains a closed halfspace of $X(\mathbf T)$, so there exist a set of positive roots $\Phi^+ \subset \Phi$ such that  for every root $\phi \in \Phi$, we have the implications
\[
v_\lambda(\phi(g)) > 0 \Rightarrow \phi \in \Phi^+ \Rightarrow v_\lambda(\phi(g)) \ge 0
\tag{***}
\]
We let $\Delta \subset \Phi^+$ be the corresponding set of simple roots.

Note also that the above group homomorphism $X(\mathbf T) \to C_\lambda/C_\lambda^-$ does not send everything to $0$, since $\Ad(g)$ has an eigenvalue $\alpha$ with $v_\lambda(\alpha) < 0$. This implies that there exists a simple root $\delta \in \Delta$ with $v_\lambda(\delta(g)) \ne 0$. In particular, the set 
\[
\Pi := \{\delta \in\Delta \mid v_\lambda(\delta(g))  = 0\}.
\]
is a proper subset of $\Delta$, so we obtain a proper parabolic subgroup $\mathbf{P}_\Pi$ of $\mathbf H$, with unipotent radical $\mathbf{U}_\Pi$. Recall that for each $\phi\in \Phi$ there is a unique Zariski closed subgroup $\mathbf{U}_\phi$ of $\mathbf{H}$ (the {\em root subgroup} corresponding to $\phi$)  and an isomorphism
$\epsilon_\phi:\mathbf{G}_a \to \mathbf{U}_\phi$ such that for each $t\in \mathbf{T}$ and $x\in K_\lambda^{\alg}$ we have $t\epsilon_\phi(x)t^{-1}=\epsilon_\phi(\phi(t)x)$.
Furthermore, by for example Proposition 8.4.3(ii) of \cite{springer}, for $\phi\in \Phi$ we have $\mathbf{U}_\phi\leq \mathbf{U}_\Pi$ if and only is $\phi$ is positive and $\phi\not\in \langle \Pi\rangle$. 

{\em Claim 4.} For $\phi\in \Phi$, we have 
$\mathbf{U}_\phi\leq \mathbf{U}_\Pi \Leftrightarrow v_\lambda(\phi(g)) > 0$.

{\em Proof of Claim 4.} Suppose first $v_\lambda(\phi(g))>0$. Then by (***), $\phi\in \Phi^+$, that is, $\phi$ is positive. Furthermore if $\phi \in \langle \Pi \rangle$ with say $\phi=\Sigma_{i\in I} \delta_i$ where the $\delta_i$ lie in $\Pi$, then
$$v_\lambda(\phi(g))=v_\lambda((\Sigma_{i\in I} \delta_i)(g))=v_\lambda(\Pi_{i\in I} (\delta_i(g)))=\Sigma_{i\in I}v_\lambda(\delta_i(g))=0,$$
(since $v_\lambda(\delta_i(g))=0$ for $\delta_i\in \Pi$), a contradiction. Thus $\phi\in \Phi^+$ and $\phi\not\in \langle \Pi\rangle$, so $\mathbf{U}_\phi\leq \mathbf{U}_\Pi$. 

Conversely, suppose that $\mathbf{U}_\phi\leq \mathbf{U}_\Pi$. Then $\phi\in \Phi^+$ as noted above, so by (***) we have $v_\lambda(\phi(g))\geq 0$. Suppose for a contradiction that $v_\lambda(\phi(g))=0,$ and write $\phi = \sum_i \delta_i$ as a sum of simple roots
$\delta_i \in \Delta$. Since all those $\delta_i$ satisfy $v_\lambda(\delta_i(g)) \ge 0$ (by (***)), the only way to obtain $v_\lambda(\phi(g)) = 0$ is that all $\delta_i$ already satisfy
$v_\lambda(\delta_i(g)) = 0$. In other words, $\phi \in \langle \Pi\rangle$, contradicting $\mathbf{U}_\phi\leq \mathbf{U}_\Pi$.
\qed \emph{(Claim~4)}

Since $\mathbf{T}$ is a $K_\lambda$-group and since 
the set $\{\phi\in \Phi \mid v_\lambda(\phi(g)) > 0\}$ is ${\rm Gal}(\Klamalg/K_\lambda)$-invariant, we deduce that 
$\mathbf{U}_\Pi$ is a $K_\lambda$-group.

{\em Claim 5.} The group $G_\lambda$ contains
$\mathbf{U}_{\Pi}(K_\lambda)$ and $\mathbf{U}^-_{\Pi}(K_\lambda)$.

{\em Proof of Claim 5.}
We prove that $G_\lambda$ contains $\mathbf{U}_{\Pi}(K_\lambda)$, the proof for $\mathbf{U}^-_{\Pi}(K_\lambda)$ being analogous.
To this end, it suffices to verify that $G_\lambda$ contains
$\mathbf{U}_{\phi}(K_\lambda)$
for each $\phi \in \Phi^+ \setminus \langle \Pi \rangle$, so fix such a root $\phi$,
and also fix an isomorphism $\epsilon_\phi\colon \mathbf{G}_a \to \mathbf{U}_\phi$ defined over $K_\lambda$, with $ t \epsilon_\phi(x)t^{-1}=\epsilon_\phi(\phi(t)x)$ for all $t\in \mathbf{T}$.
Since $\epsilon_\phi$ is continuous with respect to the valuation topology, there exists a $\lambda' \in C_\lambda/C^-_\lambda$ such that the image of the valuative ball $B' = B_{\lambda'}(0) \subset K_\lambda$ under $\epsilon_\phi$ lies in $B_\lambda(1) \cap \mathbf{U}_\phi(K_\lambda)$. 

Since $v_\lambda(\phi(g)) > 0$ (by our choice of $\phi$, and Claim 4) and since the value group $C_\lambda/C^-_\lambda$ is archimedean,
for any $a \in K_\lambda$, there exists an integer $n$ such that
$nv_\lambda(\phi(g)) + v_\lambda(a) > \lambda'$ and hence $\phi(g)^na \in B'$. 
Thus $g^n\epsilon_\phi(a)g^{-n} = \epsilon_\phi(\phi(g)^na) \in B_\lambda(1) \subset G_\lambda$, which implies $\epsilon_\phi(a) \in G_\lambda$. Since each element of $\mathbf{U}_{\phi}(K_\lambda)$ is of this form, we just proved $\mathbf{U}_{\phi}(K_\lambda) \subset G_\lambda$.
\qed \emph{(Claim~5)}

By \cite[Proposition 6.2(v)]{bortit},  $\mathbf{H}(K_\lambda)^+$
is generated by $\mathbf{U}_\Pi(K_\lambda)$ and $\mathbf{U}^-_\Pi(K_\lambda)$, so Claim~5 implies $\mathbf{H}(K_\lambda)^+ \subset G_\lambda$. This finishes the proof of (*) and hence also of the theorem.
\end{proof}

\subsection{Non-simplicity of bounded subgroups}

Our goal in this subsection is Proposition~\ref{prop:main}, which, in combination with Theorem~\ref{thm:prasadii} (or Proposition~\ref{prop:prasadiii}, in the positive characteristic case), is critical to the proof of our main theorem: It yields the unboundedness assumption that is needed to apply Theorem~\ref{thm:prasadii}.

\begin{proposition} \label{prop:main}
Suppose that $(K,v,\ldots)$ is a 1-$h$-minimal expansion of a henselian valued field of characteristic 0, or is a pure algebraically closed valued field of characteristic $p$. Let
 $G$ be an infinite  definable and definably almost simple subgroup of $\SL_n(K)$. Then $G$ is not bounded.
\end{proposition}

The proof of this proposition grew out of an argument with Bruhat-Tits buildings. While formally Bruhat-Tits buildings do not appear anymore in the proof, this is still the idea behind it. As an example of the basic idea, the group $\SL_2(\mathbb{Q}_p)$ acts on a $(p+1)$-degree tree (see Chapter II of \cite{serre-trees}), and any bounded subgroup $G < \SL_2(\mathbb{Q}_p)$ fixes a vertex and so is conjugate to a subgroup of $\SL_2(\mathbb{Z}_p)$. Using that
$\SL_2(\mathbb{Z}_p)$ has a chain of normal subgroups with trivial intersection (arising from congruence subgroups), one deduces that $G$ cannot be definably almost simple.

\begin{proof}[Proof of Proposition~\ref{prop:main}]
We suppose for a contradiction that $G$ is bounded.

Set $V=K^n$, and let $\Gamma^c$ be the set of upwards-closed subsets of $\Gamma$. Observe that $\Gamma^c$ is an ordered set (with $\lambda \le \lambda'$ iff $\lambda' \subset \lambda$, and where $\infty^c := \emptyset \in \Gamma^c$ is the maximum), and that $\Gamma$ has an order-preserving action on $\Gamma^c$ by addition. We denote that action by $+\colon \Gamma \times \Gamma^c \to \Gamma^c$, and we additionally set $\infty + \lambda := \infty^c$ for any $\lambda \in \Gamma^c$.

We define a {\em valuation} on $V$ (with values in $\Gamma^c$) to be a map 
$\eta:V \to \Gamma^c$ satisfying the following conditions for all $u,w\in V$ and $x\in K$.

\begin{enumerate}
\item[(a)] $\eta(u)=\infty^c$ if and only if $u=0$,
\item[(b)] $\eta(xu)=v(x)+\eta(u)$,
\item[(c)] $\eta(u+w)\geq\min\{\eta(u),\eta(w)\}$.
\end{enumerate}
(Here the addition in the second item is the above action of $\Gamma$ on $\Gamma^c$.)

Observe that for each $u\in V$ we have $\eta(u)=\eta(-u)$, and  that if $\eta(u_1)<\min\{\eta(u_2),\ldots,\eta(u_k)\}$ then $\eta(u_1+\ldots+u_k)=\eta(u_1)$. Indeed, to see the latter, suppose for a contradiction that $\eta(u_1+\ldots+u_k)>\eta(u_1)$, and let $w=u_2+\ldots+u_k$. Then
$$\eta(u_1)=\eta(u_1+w-w)\geq\min\{\eta(u_1+w),\eta(w)\}>\eta(u_1), \mbox{~~a contradiction.}$$

Let $\{e_1,\ldots,e_n\}$ be the standard basis of $V=K^n$. Define
the ``standard valuation'' $\eta_0$ on $V$ as $\eta_0(\sum_{i=1}^n a_ie_i) := \{\lambda \in \Gamma : \lambda \ge \min \{v(a_i):1\leq i \leq n\}\}$. We leave it to the reader to verify that $\eta_0$ satisfies the above axioms. (Note that it sends the zero vector to the empty set $\infty^c \in \Gamma^c$.)

Define $\eta\colon V \to \Gamma^c$ by $\eta(x) := \bigcup_{g \in G}\eta_0(g(x))$.
It is easily checked that $\eta$, too, is a valuation. Indeed, part (a) of the definition is clear.
To see (b), we have
$$\eta(xu)=\bigcup_{g\in G}\eta_0(g(xu))=\bigcup_{g\in G}\eta_0(xg(u))$$
$$=\bigcup_{g\in G}(v(x)+\eta_0(g(u))=v(x)+\bigcup_{g\in G} \eta_0(g(u))=v(x)+\eta(u).$$ To see (c), we have
$$\eta(u+w)=\bigcup_{g\in G}\eta_0(g(u+w))=\bigcup_{g\in G}\eta_0(g(u)+g(w))$$
$$\geq \bigcup_{g\in G}\min\{\eta_0(g(u)), \eta_0(g(w))\}=
\min\{\bigcup_{g\in G}\eta_0(g(u)), \bigcup_{g\in G}\eta_0(g(w))\}=\min\{\eta(u), \eta(w)\}.$$

Let \[\nSL_n(\eta) := \{g \in \SL_n(K) :
\forall u \in V( \eta(u) = \eta(gu))\}.\]
It is easily checked that $\nSL_n(\eta)$ is a group, and is definable. Furthermore, by the definition of $\eta$ we have $G\leq \nSL_n(\eta)$.

For $\delta\in\Gamma$ with $\delta\geq 0$, put
\[\nSL_n(\eta,\delta) := \{g \in \nSL_n(\eta):
\forall u \in V (\eta(u - gu) \ge  \delta + \eta(u))
\}.\]
Then   $\nSL_n(\eta,0)  =\nSL_n(\eta)$, and 
$\nSL_n(\eta,\delta)$  is a definable normal subgroup of $\nSL_n(\eta)$. Indeed, definability and closure under inverse is clear, and we have
(for $g,h\in \nSL_n(\eta,\delta)$) that
$$\eta(u-ghu)=\eta((u-hu)+(hu-ghu))\geq \min\{\eta(u-hu),\eta(hu-g(hu))\}$$
$$\geq \min\{\delta+\eta(u),\delta+\eta(hu))\}=\delta+\eta(u);$$
likewise, for $g\in \nSL_n(\eta,\delta)$ and $h\in \nSL_n(\eta)$ and $u\in V$ we have
$$\eta(u-h^{-1}ghu)=\eta(hu-g(hu))\geq \delta+\eta(hu)=\delta+\eta(u).$$

The intersection $\bigcap_{\delta \ge 0}\nSL_n(\eta,\delta)$ is trivial; indeed, for any $g \in \SL_n(K) \setminus \{I\}$ (where $I$ is the identity matrix), there exists an $x \in V$ with $gx \ne x$. Using boundedness of $G$, we know that $\eta(x)$ is not equal to all of $\Gamma$. Using that orbits of $\Gamma$ on $\Gamma^c$ are cofinal we deduce that there exists a $\delta$ such that $\eta(x - gx) < \delta+\eta(x)$, so $g\not\in \nSL_n(\eta,\delta)$.

Since $\bigcap_{\delta \ge 0}\nSL_n(\eta,\delta)$ is trivial, there exists a $\delta_0$ such that $N := G \cap \nSL_n(\eta,\delta_0)$ is not equal to $G$. Since $G$ is definably almost simple and $N$ is a definable normal subgroup of $G$, the group $N$ has to be finite.

By Lemma~\ref{lem:bounded-open} below (to be proved), the group
$\nSL_n(\eta,\delta_0)$ contains an open neighbourhood of the identity matrix $I$ in $\nSL_n(\eta)$. Hence,  since $N$ is finite we obtain that $G$ is a discrete subgroup of $\nSL_n(\eta)$. 
It follows by Lemma~\ref{l.locdim} that $\dim{G}=0$, and hence by Proposition~\ref{oldhyp}(i) that $G$ is finite, a contradiction.
\end{proof}

\begin{lemma} \label{lem:bounded-open}
Under the assumptions of Proposition~\ref{prop:main}, and  the assumption for contradiction that $G$ is bounded,
the group $\nSL_n(\eta,\delta_0)$ contains an open neighbourhood of the identity  $I$ of $\nSL_n(\eta)$.
\end{lemma}

\begin{proof} We first consider the case when $(K,v,\ldots)$  is a 1-$h$-minimal expansion of characteristic 0.  In this case, by Lemma~\ref{opensubgroup}(ii), it suffices to show that the tangent spaces of $\nSL_n(\eta)$ and $\nSL_n(\eta,\delta_0)$ are equal, and for this, it suffices to show that they have the same dimensions. Since the dimensions of the tangent spaces equal the dimensions of the corresponding groups, it suffices to show that $\nSL_n(\eta)$ and $\nSL_n(\eta,\delta_0)$ have the same dimension.

\begin{claim} There is $\delta_1>\delta_0$ such that for all $u\in V$ we have $\eta(u)+\delta_1>\eta(u)$.
\end{claim}

{\em Proof of Claim.}
For each $u\in V$, define $C_u:=\{\mu\in \Gamma:\mu+\eta(u)=\eta(u)\}$. Then $C_u$ is a convex subgroup of $\Gamma$, and since $G$ is bounded, $C_u\neq \Gamma$. To prove the claim, it suffices to show that there are at most $n=\dim(V)$ distinct $C_u$, for then we may choose $\delta_1>\delta_0$ outside $\bigcup_{u\in V} C_u$.

So suppose for a contradiction that $C_{u_1},\ldots, C_{u_{n+1}}$ are all distinct, where $u_1,\ldots,u_{n+1}\in V$. There are $a_1,\ldots,a_{n+1}\in K$, not all zero, such that $\sum_{i=1}^{n+1} a_iu_i=0$. Removing some terms if necessary, we may suppose that $a_i\neq 0$ for all $i$, so $C_{u_i}=C_{a_iu_i}$ for each $i$.
Thus, the $C_{a_iu_i}$ are all distinct, so the $\eta(a_iu_i)$ are all distinct. Hence, by the initial remarks on valuations in $\Gamma^c$, it follows that $\eta(\sum_{i=1}^{n+1} a_iu_i)=\min_{i=1}^{n+1} \eta(a_iu_i)\neq \infty^c$, contradicting that $\sum_{i=1}^{n+1} a_iu_i=0$.
\qed \emph{(Claim)}

Now fix $a\in K$ with $v(a)=\delta_1$. Let $M_n(K)$ be the set of $n\times n$ matrices over $K$. We define a map $\phi:\nSL_n(\eta)\to M_n(K)$ by putting
$\phi(g)=a(g-I)+I$. Since $\phi$ is injective, it remains to show that the range of $\phi$ lies in $\nSL_n(\eta, \delta_0)$ to deduce (as desired) that $\nSL_n(\eta, \delta_0)$ has the same dimension as $\nSL_n(\eta)$.

So fix $g\in \nSL_n(\eta)$ and put $h=\phi(g)$. We start by verifying that $h \in \nSL_n(\eta)$. To this end,
let $u\in V$. Observe that as $\eta(gu)=\eta(u)$, we have $\eta(gu-u)\geq \eta(u)$. Thus, we have
$$\eta(a(gu-u))=\delta_1+\eta(gu-u)\geq \delta_1+\eta(u)>\eta(u)$$
(where the last inequality holds by the choice of $\delta_1$ using the Claim). Hence,
$$\eta(hu)=\eta(a(gu-u)+u)=\min\{\eta(a(gu-u)),\eta(u)\}=\eta(u),$$
so $h\in \nSL_n(\eta).$ Furthermore,
$$\eta(hu-u)=\eta(a(gu-u)+u-u)=\eta(a(gu-u))
=\delta_1+\eta(gu-u)\geq \eta(u)+\delta_1,$$
so as desired, we obtain $h\in \nSL_n(\eta,\delta_1) \subset \nSL_n(\eta, \delta_0).$

Now consider the case when $K$ is a pure algebraically closed valued field of characteristic $p$. In this case we can no longer appeal to the arguments involving tangent spaces,
but the argument above still yields that $\nSL_n(\eta)$ and $\nSL_n(\eta,\delta_0)$ are definable groups of the same dimension. Let $\mathbf H_1 = \cl_{K^{\alg}}(\nSL_n(\eta,\delta_0))$ and $\mathbf H_2 = \cl_{K^{\alg}}(\nSL_n(\eta))$ be their  Zariski closures in $\SL_n$. Then $\mathbf H_1\leq \mathbf H_2$, and using Lemma~\ref{lou}(ii),  $\mathbf H_1, \mathbf H_2$ have the same Zariski dimension. Therefore, $|\mathbf H_2:\mathbf H_1|$ is finite.  Also, again by Lemma~\ref{lou}(ii), $\dim(\nSL_n(\eta,\delta_0))=\dim(\mathbf{H}_1(K))$.  Hence, by Lemma~\ref{interior},
$\nSL_n(\eta,\delta_0)$ has non-empty interior in $\mathbf H_1(K)$, so
 contains an open neighbourhood $U_1$ of the identity $I$ in $\mathbf H_1(K)$.
This $U_1$ is also an open neighbourhood of $I$ in $\mathbf H_2(K)$
and hence also in $\nSL_n(\eta)$.
\end{proof}

\section{Proof of Theorem~\ref{main}.}

In this section we prove Theorem~\ref{main}. The proof makes essential use of arguments from Section~\ref{sec:tan-Lie} (about Lie algebras associated to definable groups) and
classical results about simple Lie algebras over arbitrary fields, along with results from Section \ref{sec:bounded}. 
The Lie algebra approach was suggested by an argument used in the o-minimal case in \cite{pps}. 
A key issue is that in the proof of Theorem~\ref{main} we would like to take the Zariski closure of $G$ and apply Lemma~\ref{lou}(ii). However the latter requires $L_{\divv}$-definability of $G$ and is not available  in our broader  1-$h$-minimal setting. We circumvent this problem via the next result.

In the entire section, when $K$ is a field and $\mathbf H$ is a linear algebraic $K$-group, we will consider $\mathbf H$ as living in $K^{\alg}$ (\emph{i.e.}, $\mathbf H = \mathbf H(K^{\alg})$).

\begin{proposition}\label{jacobson}
Let $(K,v,\ldots)$ be a 1-$h$-minimal expansion of a henselian valued field of characteristic 0, and let $G\subseteq K^m $ be a non-abelian definable definably almost simple group with $\dim(G)=r$. Then  there is a homomorphism $\rho: G\to \SL_r(K)$ whose kernel is the (finite) center $Z(G)$ of $G$ and such that $r=\dim \rho(G) = \dim \cl_K(\rho(G))$.
\end{proposition}

\begin{proof} After replacing $G$ by $G/Z(G)$, we may assume by Propositions~\ref{p.Gman} and \ref{linearityproof} (using the adjoint representation) that $G \leq \SL_r(K)$ (where $r = \dim G$). Furthermore, via this representation, $G$ acts faithfully as a group of automorphisms of a simple Lie algebra $\mathfrak{g}$ of dimension $r$ over $K$. We must show
$\dim(G) = \dim\cl_K(G)$.

We  use some basic theory concerning simple Lie algebras over arbitrary fields, and their automorphism groups, taken from Chapters IX and X of Jacobson \cite{jacobson}.
First, using that $\mathfrak{g}$ is simple, there is a canonical finite extension field $F$ of $K$, called the \emph{centroid} of $\mathfrak{g}$, with the property that $\mathfrak{g}$ has the structure of a Lie algebra over $F$, and for every extension field $F'$ of $F$, the Lie algebra $ \mathfrak{g}\otimes_F F'$ (over $F'$) is simple. To construct $F$, define $A$ to be the associative  $K$-algebra of all $K$-linear maps $\mathfrak{g} \to \mathfrak{g}$.
Let $T(\mathfrak{g})$ be the $K$-subalgebra of $A$ generated by all  the maps
$a_L\colon x \mapsto [a,x]$ and $a_R\colon x \mapsto [x,a]$ for $a\in \mathfrak{g}$. Then $F$ is the centraliser of $T(\mathfrak{g})$ in $A$.  Since $A$ has finite vector space dimension over $K$, it follows that $F$ is finite dimensional over $K$, so is a finite degree extension of $K$, of degree $d$ say. 
Observe that 
$\ldim_K(\mathfrak{g})=d\cdot\ldim_F(\mathfrak{g})$, where $\ldim_K$ and $\ldim_F$ denote respectively vector space dimension over $K$ and $F$. We shall write $\Aut_K(\mathfrak{g})$ for the group of automorphisms of $\mathfrak{g}$ viewed as a Lie algebra over $K$, and $\Aut_F(\mathfrak{g})$ for the group of automorphisms of $\mathfrak{g}$ viewed as a Lie algebra over $F$.

By Theorem 5 of Ch.~X of Jacobson \cite{jacobson}, the group $\Aut_K(\mathfrak{g})$ is semilinear over $F$, that is, if $g\in \Aut_K(\mathfrak{g})$ then there is $\sigma \in \Aut(F/K)$ such that $g(ax)=a^\sigma g(x)$ for $a\in F$ and $x\in \mathfrak{g}$. Since $\Aut(F/K)$ is finite, $\Aut_F(\mathfrak{g})$ is a normal subgroup of $\Aut_K(\mathfrak{g})$ of finite index.
Let $G'$ be the group $\Aut_F(\mathfrak{g})$, considered as a subgroup of $\GL_r(K)$ which is definable in our structure $K$.
Since $G$ is definably almost simple, it follows that $G \leq G'$, and since $G'$ is Zariski closed in $\GL_r(K)$, to obtain $\dim(G) = \dim \cl_K(G)$, it remains to prove that $\dim G' \le r$.

Let $m$ be the vector space dimension of $\mathfrak{g}$ over $F$, so $md=r=\dim(G)$.
Using any identification of the $F$-vector space $\mathfrak{g}$ with $F^m$, we can consider $G'=\Aut_F(\mathfrak{g})$ as a subgroup of $\GL_m(F)$. As such, it is the set $\mathbf H(F)$ of $F$-rational points of an $F$-group $\mathbf H \subset \GL_m$. Let $e := \zardim \mathbf H$ be its Zariski dimension. Then $\dim G' \le de$. Indeed, by Noether normalization, there exists a finite-to-one $F$-morphism $\mathbf H \to (K^{\alg})^e$. This induces a finite-to-one map from $G' = \mathbf H(F)$ to $F^e$. Composing this with a $K$-linear map $F \to K^d$ yields a finite-to-one map $G' \to K^{d e}$ which is definable in the structure $K$. This implies the inequality $\dim(G') \le de$, as claimed. Thus, to obtain $\dim G' \le r$, it remains to prove that $m = e$.

We now pass to $K^{\alg}$, as follows:
Since $\mathfrak{g}\otimes_F F'$ is a simple Lie algebra over $F'$ for every $F' \supset F$, also $\mathfrak{g}\otimes_F K^{\alg}$ is a simple Lie algebra over $K^{\alg}$. Note that we have $\mathbf H = \Aut_{K^{\alg}}(\mathfrak{g}\otimes_F K^{\alg})$ (since preservation of the Lie bracket of $\mathfrak{g}\otimes_F K^{\alg}$ can be verified on an $F$-basis of $\mathfrak{g}$).
The connected component $\mathbf H^\circ$ is a simple algebraic group of the same Lie type as $\mathfrak{g}\otimes_F K^{\alg}$, so we have $e = \zardim \mathbf H = \zardim \mathbf H^\circ = \ldim_{K^{\alg}} (\mathfrak{g}\otimes_F K^{\alg}) = m$.
\end{proof}

Finally, we prove our main theorem, which we first restate.

\begin{theorem} \label{main2} 
\begin{enumerate}
\item[(i)] Let $(K,v,\ldots)$ be a 1-$h$-minimal  expansion of a henselian valued field of characteristic 0.
Let $G$ be a non-abelian group definable  in the structure $(K,v,\ldots)$ with universe a definable subset of $K^n$ for some $n$, and suppose that $G$ is definably almost simple.
Then there is a semisimple, almost $K$-simple and $K$-isotropic linear algebraic $K$-group $\mathbf{H}$ and a group $G^*$ definably isomorphic to $G/Z(G)$, 
 such that 
$\mathbf{H}(K)^+\leq G^*\leq \mathbf H(K)$. 
\item[(ii)] Let $(K,v)$ be an algebraically closed valued field of characteristic $p>0$, and let $G\leq \SL_n(K)$
be a definable non-abelian group which is definably almost simple. Then there is a semisimple, almost $K$-simple and $K$-isotropic linear algebraic $K$-group $\mathbf{H}$ such that
$G/Z(G)$ is definably isomorphic to $\mathbf{H}(K)$.

\end{enumerate}
\end{theorem}

\begin{proof}

 Let $r=\dim(G)$. Suppose first that $K$ has characteristic 0. After replacing $G$ by $G/Z(G)$, we may assume by Proposition~\ref{jacobson} that $G\leq \SL_r(K)$ and that   $r=\dim G = \dim \cl_K(G)$.

 In characteristic $p$, the group $G$ is assumed to be an $L_{\divv}$-definable subgroup of some $\SL_n(K)$, and then automatically by Lemma~\ref{lou}(ii) we have that $\dim G = \dim \cl_K(G)$. Using \cite[Theorem 6.8]{borel} to factor out the centre, we may suppose that $G$ is centreless. 

Thus, either way, we may assume that $G\leq \SL_n(K)$ with $\dim G = \dim \cl_K(G)$, and that $G$ is centreless.
In particular, definable almost simplicity of $G$ already implies definable simplicity (see Remark~\ref{rem:simple}).

\smallskip

\emph{Claim 1.} There is a semisimple, almost $K$-simple linear algebraic $K$-group $\mathbf{H}$ and a definable injective group homomorphism $\phi\colon G \to \mathbf H(K)$  such that for $G^* := \phi(G)$, the following holds:
the Zariski closure $\cl_K(G^*)$ (of $G^*$ in $\mathbf H(K)$) is all of $\mathbf H(K)$, and $\dim(G^*) = \dim(\mathbf{H}(K))$.

\begin{proof}
Let $\tilde{\mathbf H} := {\rm cl}_{K^{\alg}}(G)$
be the Zariski closure of $G$ in $\SL_n$.
Then
$\tilde{\mathbf H}$ is a subgroup of $\SL_n$ (for example by \cite[Ch.~I, Proposition 1.3(b)]{borell}). Since any
field automorphism $\alpha$ of $K^{\alg}$ which fixes $K$ pointwise also fixes $G$ (as its universe lies in a power of $K$), we also obtain that $\tilde{\mathbf H}$ is a $K$-group, and from that, we deduce that in particular, we have $\cl_K(G) = \tilde{\mathbf H}(K)$; see also Remark~\ref{note-on-dim}.

Clearly, $\tilde{\mathbf H}$ is connected.
Indeed, its connected component $\tilde{\mathbf H}^\circ$
is normal of finite index in $\tilde{\mathbf H}$, so 
$G\cap \tilde{\mathbf H}^\circ$ is a definable normal subgroup of $G$ of finite index,
so contains $G$ by definable  simplicity of $G$. Thus $\cl_{K^{\alg}}(G) \subset \tilde{\mathbf H}^\circ$ and hence we have equality.
 
Let $\mathbf{N}$ be a connected proper normal $K$-subgroup of $\tilde{\mathbf H}$ of maximal Zariski dimension. Then $\mathbf{N}(K)\cap G$ is a definable normal subgroup of $G$. If $\mathbf{N}(K)\cap G = G$, then $\mathbf{N}=\tilde{\mathbf H}$ contradicting properness of $\mathbf{N}$ in $\tilde{\mathbf H}$. Thus, by definable simplicity of $G$, 
$\mathbf{N}(K)\cap G=1$.

Let $\mathbf{H} = \tilde{\mathbf H}/\mathbf{N}$ be a quotient in the category of algebraic $K$-groups. Then $\mathbf{H}$ is connected and almost 
$K$-simple (by maximality of $\mathbf{N}$), and is a linear algebraic group defined over $K$, by \cite[Theorem 6.8]{borel}. 
Let
\[\rho\colon \tilde{\mathbf H} \to \mathbf{H}\] be the canonical $K$-epimorphism with $\mathbf{N}=\Ker(\rho)$.
In particular, $\rho:\tilde{\mathbf H}\to \mathbf{H}$ is surjective with kernel $\mathbf{N}$.

Let $\phi:=\rho_{|G}$, and put $G^*:=\phi(G)$. Note that $G^*$ is Zariski dense in $\mathbf{H}$, as $G$ is Zariski dense in $\tilde{\mathbf H}$.
Since $\Ker(\phi)=G\cap \mathbf{N}(K)$ is trivial, by Proposition~\ref{oldhyp}(ii) we have $\dim(G) = \dim(G^*)\leq 
\dim(\mathbf{H}(K)) \leq \dim(\tilde{\mathbf{H}}(K))$. Also, recall that we already established before Claim~1 that 
$\dim(G)=\dim \cl_K(G) = \dim \tilde{\mathbf H}(K)$. It follows that all these dimensions are equal, and in particular
$$\dim(G^*)=\dim(\mathbf{H}(K)).$$

The solvable radical $\Rr(\mathbf{H})$ of $\mathbf{H}$ is $K$-closed and connected, so (since $K$ is perfect) $\Rr(\mathbf{H})$ is defined over
$K$ (see \cite[5.1]{borell}), and hence is trivial, by almost $K$-simplicity of $\mathbf{H}$. This gives semisimplicity of $\mathbf{H}$, and hence the claim.
\qedhere \emph{(Claim 1)}
\end{proof}

Fix any $K$-group embedding of $\mathbf H$ into some $\SL_m$.
Then by Proposition~\ref{prop:main}, $G^*$ is not bounded, and as $G^*\leq \mathbf{H}(K)$, the group $\mathbf{H}(K)$ is not bounded either. Hence, by Theorem~\ref{thm:prasadi}, $\mathbf{H}$ is isotropic over $K$.

Since $\dim(G^*)=\dim(\mathbf{H}(K))$, it follows from Lemma~\ref{interior} 
 that $G^*$ has non-empty interior in 
$\mathbf H(K)$. 
Since $G^*$ is a subgroup of $\mathbf H(K)$, having non-empty interior already implies that $G^*$ is open in $\mathbf H(K)$ (pick an open subset of $G^*$ and translate it around).
As noted above, $G^*$ is not bounded. 

In the case when $(K,v,\ldots)$ is 1-$h$-minimal of characteristic 0, it follows by Theorem~\ref{thm:prasadii} that $G^*\geq \mathbf{H}(K)^+$, so we are done. In the case when $(K,v)$ is a pure algebraically closed valued field of characteristic $p$, it follows from Proposition~\ref{prop:prasadiii} that $G$ is not a proper subgroup of $\mathbf{H}(K)^+$.
Since $K$ is algebraically closed (and $\mathbf H$ is semisimple), we have $\mathbf{H}(K)^+ = \mathbf{H}(K)$ (by Remark~\ref{rem:g+acf}),
and so $G=\mathbf{H}(K)$.
This completes the proof of the theorem.
\end{proof}
Competing interests: the authors have none.

\bibliography{groups}
\bibliographystyle{alpha}

\end{document}